\documentclass[a4paper]{amsart}

\usepackage{amsthm, amsmath, amssymb, bbm, stmaryrd,mathtools}
\usepackage[T1]{fontenc}
\usepackage{mathrsfs}
\usepackage{tikz}
\usepackage{tikz-cd}
\usetikzlibrary{decorations.markings}
\usetikzlibrary{calc} 
\usetikzlibrary{through}
\usetikzlibrary{backgrounds}
\usepackage{bbold}

\usepackage{xy} \xyoption{all}
\newdir{ >}{{}*!/-10pt/@{>}}

\usepackage{color}
\usepackage{enumerate}

\usepackage{Stokeshypergeo}

\begin{document}

\title[Stokes matrices for confluent hypergeometric equations]{Stokes matrices for confluent hypergeometric equations}

\author[M. Hien]{Marco Hien}
\address[M.~Hien]{Institut f\"ur Mathematik, Universit\"at Augsburg, 86135 Augsburg, Germany}
\email{marco.hien@math.uni-augsburg.de}

\begin{abstract}
We apply the method of \cite{DHMS} to compute the Stokes matrices of non-resonant confluent hypergeometric differential equations. We discuss the ambiguity of the presentation of the Stokes matrices regarding different choices. The results rely on an explicit description of the perverse sheaf associated to the non-confluent regular singular hypergeometric system arising via Fourier-Laplace transform. We give assumptions on the parameter such that the Stokes matrices have rational or real values. Under some more restrictive conditions, the Stokes matrices had been computed by Duval-Mitschi before. We compare our results with their formulae in the unramified case. 
\end{abstract}

\maketitle

\tableofcontents

\section{Introduction}

Hypergeometric differential equations have been widely studied as explicit examples of meromorphic differential equations with possibly irregular singularities (in the \emph{confluent} case) providing analogies to wild ramification phenomena in the theory of $\ell $-adic sheaves. In the notation of Katz, the \emph{hypergeometric equation} for the parameters $\ualpha \in \C^n $, $\ubeta \in \C^m $ and $\glambda \in \C^\times $ is defined as
\begin{equation}\label{eq:defHyp}
\opHypabl:=\glambda \cdot \prod_{i=1}^n (z\del_z-\alpha_i) - z \cdot \prod_{j=1}^m (z\del_z -\beta_j) \ ,
\end{equation}
where $z $ is the complex coordinate. Nicholas Katz presented a vast investigation of these equations from the point of view of D-module theory (see  \cite{katz})-- one of his main goals being the understanding of the associated differential Galois group, its analogies to the $\ell $-adic world and applications to exponential sums over finite fields. In a generic situation (regarding the parameters), these equations are irreducible and even rigid (see \cite{blochesnault}).

Lately, there has been interest in the hypergeometric equation \eqref{eq:defHyp} regarding Hodge theoretic properties. If $n=m $, the equation \eqref{eq:defHyp} has regular singularities at $\{ 0, \glambda, \infty \} $. If $n>m $, the singularities are $\{0, \infty \} $ and $z=\infty $ is an irregular singularity. In the regular singular case, there are results by R. Fedorov \cite{fedorov} on the Hodge numbers of a natural variation of Hodge structures it underlies. In the irregular case, assuming the parameters are real, there is an underlying irregular Hodge structure in the sense of C. Sabbah \cite{sabbahHodge}. The corresponding irregular Hodge numbers are determined by C. Sabbah and J-D. Yu \cite{sabyu} -- special cases had been obtained before by A. Casta\~no Dom\'inguez, Th. Reichelt and Ch. Sevenheck (\cite{albrThomasSev1}, \cite{albrThomasSev2}).
 
In this article, we investigate the Stokes phenomenon of the confluent hypergeometric equation for $n>m $ at infinity. Since $z=\infty $ is an irregular singular point of \eqref{eq:defHyp}, the local isomorphism class of the equation at this point is determined by its Stokes structure in the sense of P. Deligne and B. Malgrange (see \cite{Mal91}). It is well-known (\cite{katz}), that the slopes at infinity are $0 $ and $1/d $ with $d\defeq n-m $. Hence, the pull-back with respect to the ramification map $y \mapsto y^d=z $ has slopes $0 $ and $1 $. The general theory (\cite{Mal91}, \cite{BJL79}) yields that the Stokes structure of the latter can be encoded in two complex matrices $S_+, S_- $ (the \emph{Stokes matrices} defined as the transition matrices of asymptotic solutions in two sectors of width $\pi + \varepsilon $ centered at $z=\infty $), one of them upper, the other lower triangular. The explicit presentation of these matrices involves several choices. 

As usual, we will assume a genericness assumption on the parameters $\alpha_j $ and $\beta_i $. If one moreover restricts to the case where all $\beta_j $ are pairwise non-equivalent modulo the integers, these Stokes matrices for $\opHypl{\ualpha}{\ubeta}{-1} $ have already been computed directly according to their definition (and therefore by producing the asymptotic solutions of \eqref{eq:defHyp} explicitly) by A. Duval and C. Mitschi in \cite{DM}. Their result includes the values of the Gamma-function on certain combinations of the parameters. The authors do not address the question of ambiguity of the Stokes matrices, but they fix the choices made in the procedure -- the given basis of formal solutions and their asymptotic lifts produced by standard methods. 

We propose a different way to obtain the Stokes matrices of \eqref{eq:defHyp} in the generic case (but without the more restrictive assumption). Our method is based on the result of A. D'Agnolo, G. Morando, C. Sabbah and the author in \cite{DHMS} in combination with a Theorem by N.~Katz. The latter represents the confluent case as the Fourier transform of a non-confluent one. The result of \cite{DHMS} gives the Stokes matrices of the Fourier transform of a regular singular D-module on the affine line once the associated perverse sheaf of solutions is sufficiently known. More precisely, this perverse sheaf can be described by its quiver (an object of linear algebra) and \cite{DHMS} directly gives a formula for the Stokes matrices in terms of the quiver. 

Usually, a major difficulty in determining the quiver of the perverse sheaf of solutions of a regular singular D-module lies in the fact that the information on the global monodromy of the local system of solutions away from the singularities is needed. The latter can be computed by solving the equation locally and studying its analytic continuation along paths -- a difficult task. In the case of the regular singular hypergeometric equation, under some non-resonance condition, there is a beautiful observation due to Levelt (about which we learnt from its application in the work of F. Beukers and G. Heckman (\cite{BH}) on the monodromy of the hypergeometric equation). We know that the local system is rigid, hence its global monodromy as a representation of the fundamental group is determined (up to conjugation) by the (conjugacy classes of) the individual local monodromies. A priori, this knowledge alone does not allow to determine the global monodromies explicitly from the local ones. Levelt's Lemma (Lemma \ref{lemma:Levelt} below) however produces such an explicit representation in the case of the regular singular hypergeometric equation. Since the non-resonance condition involves that the equation is irreducible, we know that the perverse sheaf is the middle extension of the local system and therefore, we can determine its quiver.

Our main result is the computation of the Stokes matrices for $\opHypabl $ for generic $\ualpha, \ubeta $ and any $\glambda \in \C^\times $. In the unramified case $d=1 $, we give two presentations, the first in companion form:
Let $\ualpha \defeq (\alpha_1, \ldots, \alpha_n) $ and $\ubeta \defeq (\beta_1, \ldots, \beta_{n-1}) $ be generic and consider the polynomials
\begin{align*}
\chi_B(X) \defeq \prod_{j=1}^{n-1} (X-\exp(-2\pi i \beta_j)) &= X^{n-1}+B_1 X^{n-2}+ B_2 X^{n-3} + \ldots + B_{n-1}, \\
\chi_A(X) \defeq \prod_{j=1}^n (X- \exp(-2\pi i \alpha_j)) &= X^n + A_1 X^{n-1} + A_2 X^{n-2}+ \ldots + A_n.
\end{align*}
Let us put $B_n \defeq 0 $. We denote by $\Co{-\ubeta} $ the companion matrix associated to the polynomial $\chi_B(X) $ -- see \eqref{eq:companiondef} and \eqref{eq:Cgamma}. We will write $1_k $ for the identity matrix of size $k \times k $.
\begin{theorem}[=Theorem \ref{thm:companion}]
For $d=1 $, the hypergeometric system $\Hypl{\ualpha}{\ubeta}{\glambda} $ at infinity is represented by the pair
\[
 S_+  = 
 \left(
 \begin{array}{c|c}
  1_{n-1} & x \\[.2cm] \hline
  0 & 1
\end{array}\right)
  \text{ and }
   S_- = 
  \left(
 \begin{array}{c|c}
  \mbox{$\Co{-\ubeta}$} & 0 \\[.2cm] \hline
y & \exp(2\pi i \lambda)
\end{array}\right)
\]
where $y = 
\big( (-1)^n\exp(2\pi i \sum_{j=1}^{n-1} \beta_j) ,0 \ldots, 0 \big) $, $\lambda = 1- \sum_{j=1}^n \alpha_j + \sum_{j=1}^{n-1} \beta_j $,
and $x={}^t(x_1, \ldots, x_{n-1}) $ with
\[
x_j =  A_{j+1}-B_{j+1} - (A_1-B_1) \cdot B_j
\] 
for $j=1, \ldots, n-1 $
\end{theorem}

The choices involved in this presentation by the pair $[S_+,S_-] $ are discussed in section \ref{sec:ambiguity}. Other than in the work of Duval-Mitschi, we do not assume that the $\beta_j $ are pairwise non-equivalent modulo the integers. We also compute a variant of the main result, where we give the Stokes matrices in terms of the Jordan blocks associated to the eigenvalues $\exp(-2\pi i \beta_j) $ and their multiplicities. 

\begin{theorem}[=Theorem \ref{thm:jordan}]
Assume $d=1 $. If we subdivide the roots of $\chi_B $ as
\[
\{ \lambda_1, \ldots, \lambda_\ell \} = \{ e^{-2\pi i \beta_1},\ldots,  e^{-2\pi i \beta_{n-1}} \}
\]
with pairwise different $\lambda_j $, and we denote by $\kappa_j$ the corresponding multiplicity, so that $\chi_B(X)=\prod_{j=1}^\ell (X- \lambda_j)^{\kappa_j} $, the equivalence class of Stokes matrices for the hypergeometric system $\Hypl{\ualpha}{\ubeta}{\glambda} $ at infinity is given by its normal form ($M_\infty^- $ being the formal monodromy -- Notation \ref{notation:formon}):
\[
[S_+,S_-] = \left[
\left( 
\begin{array}{c|c}
1_{n-1}  & 
z \\ \hline
0 &  1
\end{array}\right) \ , \ 
M_\infty^- \cdot 
\left( 
\begin{array}{c|c}
1_{n-1}  & 0 \\ \hline
\fre & 1
\end{array}\right) 
\right]
 \]
with $\fre=(1,0, \ldots, 0 \mid \ldots \mid 1,0, \ldots 0) \in \mathrm{Hom}(\C^n,\C) $ with blocks of size $\kappa_1, \ldots, \kappa_{\ell} $ and $
z = e^{-2\pi i \lambda} \cdot~{}^t(z_1, \ldots, z_\ell) $
with
\begin{equation*}
z_j = \TaylX{\frac{\chi_A(X)}{X}  \cdot \frac{(X-\lambda_j)^{\kappa_j}}{\chi_B(X)}}{\kappa_j}{\lambda_j} \text{\quad for $j=1, \ldots, \ell$ -- see Notation \ref{not:taylor}.} 
\end{equation*}
\end{theorem}
Note that a representative in \emph{normal form} (see Proposition \ref{prop:xy}) as in the result is uniquely determined by its class. We discuss cases where one obtains presentations with real or integer coefficients in the Stokes matrices -- section \ref{sec:cyclo}. In the diagonalizable case, we find the situation considered by Duval-Mitschi and we prove that our result is equivalent to the one of Duval-Mitschi.

In section \ref{sec:ramified}, we extend our main result to the ramified case $d \ge 2 $. We compute a representative of the Stokes matrices for the pull-back $\jint[d]^\ast \Hypabl $ in companion form -- Theorem \ref{thm:companionram}. The combinatorics of the presentation is more involved as in the unramified case. We deduce analogous statements on cases where the Stokes matrices have rational or real entries.

Note that the Stokes matrices of the Kummer pull-backs of hypergeometric systems arising from the quantum cohomology of the projective space were computed by D. Guzzetti in \cite{guz}. The hypergeometric systems in his work have parameters $\alpha_j=0 $ and hence do not satisfy the genericness assumptions in our approach.

\section{Confluent hypergeometric systems}

Let $z $ denote the coordinate in $\Af^1 $ and $\D:=\C[z,z^{-1}]\langle \partial_z \rangle $. Let $\ualpha=(\alpha_1, \ldots, \alpha_n) \in \C^n $ and $\ubeta=(\beta_1, \ldots, \beta_m) \in \C^m $ be fixed parameters. Furthermore, let $\glambda \in \C^\times $. We want to study the hypergeometric differential operator
\[
\opHypabl:=\glambda \cdot \prod_{i=1}^n (z\del_z-\alpha_i) - z \cdot \prod_{j=1}^m (z\del_z -\beta_j) \ .
\]
We denote by $\Hypabl:= \D/\D \opHypabl $ the corresponding module. It defines an algebraic D-module on the multiplicative group $\Gm $ which we denote with the same symbol. If $\glambda=1 $, we omit the subscript. The following are well-known facts (e.g. \cite{katz}):
\begin{enumerate}
 \item For $n=m $, the singularities of $\Hypabl $ are $\{0,\glambda,\infty \} $, and the module is regular singular at each of them.
 \item For $n \neq m $ -- the \emph{confluent case} -- the singularities are $\{0, \infty \} $. If $n>m $, the module is regular singular at $0 $ and irregular singular at $\infty $. For $n<m $, the opposite is true. 
\item Let us consider the confluent case with $n>m $. Then the irregular singularity at $\infty $ is ramified of order $d $.
% and the Levelt-Turrittin decomposition reads as
%\[
%[d]^\ast \Hypab \otimes_{\C[u]} \CL{u} \simeq
%\]
%where $y:=\tfrac{1}{z} $ is the coordinate at infinity and $[d]:u \mapsto y=u^d $ is the d-fold covering map.
\end{enumerate}
\begin{definition}
 The set of parameters $\uab \in \C^{n+m} $ is called \emph{non-resonant} if $\alpha_i - \beta_j \not\in \Z $ for all $i,j $.
\end{definition}

\begin{lemma}[{\cite[Corollary 3.2.1]{katz}}] 
The hypergeometric module $\Hypabl $ is irreducible if and only if $\uab $ are non-resonant. 
\end{lemma}

We will make use of the following Theorem due to N. Katz describing the confluent case in terms of the Fourier transform of a non-confluent hypergeometric system. Let us introduce some notation first.   We denote by $j:\Gm \to \Af^1 $ the inclusion. Recall that $\jint $ denotes the \emph{middle extension functor} for D-modules (cp. \cite[2.9]{katz}).

Let $\Af^1 $ be the dual affine line which we endow with the coordinate $t $. The Fourier transform of a $\Wz $-module $M $ is the $\Wt $-module $\Fou{M} $ given by the same $\C $-vector space $\Fou{M}=M $ with the action of $t $ and $\del_t $ given by $t \cdot m := \del_z m \text{ and } \del_t m := -z \cdot m $.

\begin{theorem}[{\cite[Theorem 6.2.1]{katz}}] \label{thm:katzFourier}
 If $\uab $ are non-resonant, additionally $d\alpha_i \not\in \Z $ for all $i=1, \ldots, n $, and $\Hypabl $ is not Kummer induced (see below), then
 \[
 \jint[d]^\ast \Hypabl \simeq \Fou{\big(\jint [d]^\ast \Hypl{\frac{1}{d}, \frac{2}{d}, \ldots, \frac{d}{d}, - \ubeta}{-\ualpha}{(d^d)/\glambda} \big)}
 \]
\end{theorem}
Note that the hypergeometric module on the right hand side is of type $(n,n) $, hence regular singular at $0, \rglambda \defeq (d^d)\glambda^{-1}, \infty $.

\begin{remark}
A hypergeometric system is called Kummer induced if it is the push-forward of another hypergeometric system with respect to the $d $-fold ramification map $z \mapsto z^d $ on $\Gm $. The Kummer Recognition Lemma 3.5.6 in \cite{katz} does exactly what its name says. In particular, $\Hypabl $ cannot be Kummer induced if $\text{gcd}(n,m)=1 $, e.g. if $m=n-1 $.
\end{remark}

In the first part of the article, we will consider the unramified case for generic parameters. In section \ref{sec:ramified} we extend the results to the ramified case.

\begin{assumption}\label{ass:generic}
We assume that $m=n-1 $, $\uab $ are non-resonant and $\alpha_i \not\in \Z $ for all $i=1, \ldots, n $. We will call the parameters satisfying this assumption to be \emph{generic}. 
\end{assumption}

Note that $\ualpha=(\alpha_1, \ldots, \alpha_n) $ and $\ubeta=(\beta_1, \ldots, \beta_{n-1}) $. Katz's Theorem then tells us that
\begin{equation}\label{eq:FouH1}
\jint \Hypabl \simeq   \Fou{\big( \jint \Hypl{1, -\ubeta}{-\ualpha}{\glambda^{-1}} \big)} \ .
\end{equation}
We will use \eqref{eq:FouH1} for our computation of the Stokes matrices at infinity of the left hand side.

\begin{remark}\label{rem:statphas}
By the stationary phase result of Bloch-Esnault \cite{blochesnault} and Garcia-Lopez \cite{garc}, we deduce from \eqref{eq:FouH1} that the exponential factors in the formal decomposition (Levelt-Turrittin decomposition) of $\Hypabl $ at $z=\infty $ are $1=e^{0 z} $ and $e^{\glambda^{-1} z} $ determined by the singularities of its regular singular inverse Fourier transform. 
\end{remark}

\section{Stokes matrices} \label{sec:Stokesgeneral}

The local isomorphism class of $\Hypabl $ at infinity can be described by Stokes matrices (sometimes called Stokes multipliers). We shortly recall their definition and the approach used in \cite{DHMS} based on D'Agnolo-Kashiwara's Riemann-Hilbert correspondence.

\subsection{Definition of the Stokes matrices}\label{sec:dhms}

Usually the local classification proceeds as follows. After formal completion (and ramification of the coordinate), the connection is isomorphic to a direct sum of elementary exponential connections twisted by regular singular connections. Let us consider a meromorphic connection $M $ at infinity of exponential type, i.e. assuming
\begin{equation}\label{eq:LT}
\formalz{M} \simeq \bigoplus_{c \in \Sigma}  \big(E^{c z} \otimes R_c \big)
\end{equation}
for some finite subset $\Sigma \subset \C $ where $E^{c z} $ is the rank one connection $\nabla=d-d(c z) $ (with solution $e^{c z} $) and $R_c $ is regular singular. Note that $0 \in \Sigma $ is allowed. 

Choosing a nearby point $x \neq \infty $, each $R_c $ is uniquely determined (up to isomorphism) by the stalk $\Psi_c \defeq \shs ol(R_c)_x $ of its solutions sheaf at $x $ and the monodromy $T_c$. The linear isomorphism
\[
\bigoplus_{c\in \Sigma} T_c \in \mathrm{Aut}(\bigoplus_{c \in \Sigma}\Psi_c)
\]
consisting of the diagonal blocks is the \emph{formal monodromy} of $M $. 

As a second step, one considers formal solutions of $M $ using \eqref{eq:LT} and looks for asymptotic lifts (following the work of Hukuhara, Malgrange and Sibuya). Asymptotic lifts exist on sectors of width $\pi+\varepsilon $ (Balser, Jurkhat, Lutz). The Stokes matrices are then by definition the transition matrices of the asymptotic solutions on the intersection of the sectors -- see for example Definition 3.4 of \cite{DM}. 

\begin{remark}
The definition of the Stokes matrices (as e.g. in \cite{DM}, \cite{cotti}) hinted on above requires several choices. Usually, the question of ambiguity is not addressed. The reason is that in the situations under consideration there are standard constructions for the formal solutions and standard procedures to produce asymptotic lifts (Borel (multi-)summation method). Additionally, the formal monodromy usually is separated from the Stokes matrices and the latter are then required to have the identities along the block diagonal. We will come back to this in subsection \ref{subsec:equrel}.
\end{remark}

In \cite{DK16} A. D'Agnolo and M. Kashiwara introduce the category of enhanced ind-sheaves and prove a Riemann-Hilbert correspondence for holonomic D-modules in any dimension. In \cite[9.6]{DK16} they describe how the Stokes matrices are encoded in the enhanced ind-solution sheaf. The Stokes matrices then are defined to be the transition matrices of the associated enhanced solutions sheaves on the sectors. 

The result of \cite{DHMS} on which our computations will rely takes up this point of view. Let $\pi:\Af^1 \times \R \to \Af^1 $ be the projection -- the enhanced sheaves live on $\Af^1 \times \R $. If $K \defeq \solE(M) $ is the enhanced solutions ind-sheaf of $M $ (see \cite{DK16}), the formal decomposition and the asymptotic lifting property induce isomorphisms
\[
\pi^{-1} \C_S \otimes K \isoto \pi^{-1} \C_S \otimes \bigoplus_{c\in \Sigma} \big( \mathbb{E}^{c z} \otimes \Psi_c \big)
\]
in D'Agnolo-Kashiwara's category $\enh^{\mathrm{b}}_\Rc(\iCfield_{\Af^1}) $ for a sector $S \subset \Af^1 $ where this lift exists. Here $\mathbb{E}^{c z} \defeq \solE(E^{c z}) $ is the enhanced solutions sheaf of the exponential connection. Note that $\mathbb{E}^0 $ is equal to $\C^\enh_{\Af^1} $ in the notation of \cite{DK16}.

The choice of the sectors and orientation is encoded by a fixed pair $\fra,\frb \in \C^\times $ (in \cite{DHMS} these are denoted $\alpha, \beta $) such that
\begin{align}\label{eq:alphbet}
&\Re(\fra \cdot \frb) = 0, \\ \notag
&\Re((c-c') \cdot \frb) \neq 0, \quad\forall c, c'\in \Sigma,\ c\neq c'.
\end{align}
The latter induces the ordering
\begin{equation}\label{eq:orderingfrb}
c<_\frb c' \Longleftrightarrow \Re(c \frb) < \Re(c' \frb).
\end{equation}

According to \cite{DHMS}, we consider the (closed) sectors $H_{\pm \fra} \defeq \{ z \in \Af^1 \mid \pm \Re (\fra z) \ge 0 \}$. If we write $h_{\pm \frb} \defeq \pm \R_{>0} \frb $ for the half-lines with direction $\pm\frb $, we have $H_\fra \cap H_{-\fra} =h_{\frb} \cup h_{-\frb} $. We then get isomorphisms
\begin{equation} \label{eq:salpha}
s_{\pm \fra}: \pi^{-1} \C_{H_{\pm \fra}} \otimes K 
 \isoto 
\pi^{-1} \C_{H_{\pm \fra}} \otimes \bigoplus_{c \in \Sigma} \big(\mathbb{E}^{c z} \otimes \Psi_c \big)
\end{equation}
We have both isomorphisms over the intersections and hence we can define the transition isomorphisms as
\begin{align}
&\widetilde{s}_{\pm \frb} \in \mathrm{Aut}\big( 
\pi^{-1} \C_{h_{\pm \frb}} \otimes \bigoplus_{c\in \Sigma} (\mathbb{E}^{c z} \otimes \Psi_c)\big), \notag \\ \label{eq:Stokesdef}
& \widetilde{s}_{\pm \frb} \defeq \big( \pi^{-1} \C_{h_{\pm \frb} }\otimes s_{-\fra} \big) \circ 
\big( \pi^{-1} \C_{h_{\pm \frb}} \otimes s_{+\fra} \big)^{-1}.
\end{align}
According to \cite[9.8]{DK16} (see also \cite[Lemma 5.2]{DHMS}), for any $c,c' \in \C $ and vector spaces $\Psi_c $ and $\Psi_{c'} $, we have
\begin{multline*}
{\Hom}_{\enh^{\mathrm{b}}_\Rc(\iCfield_X)}
(\pi^{-1} \C_{h_{\frb}} \otimes \mathbb{E}^{c z} \otimes \Psi_c,
\pi^{-1} \C_{h_{\frb}} \otimes \mathbb{E}^{c' z} \otimes \Psi_{c'}) \simeq \\
\left\{
\begin{array}{ll}
{\Hom}_{\C}(\Psi_c,\Psi_{c'}) & \text{if $\Re(c\frb)>\Re(c' \frb)$}\\
 0 & \text{if $\Re(c\frb)<\Re(c'\frb)$}
\end{array}\right.
\end{multline*}
(similarly for $-\frb $ instead of $\frb $) and
\begin{equation}\label{eq:t}
\mathrm{End}_{\enh^{\mathrm{b}}_\Rc(\iCfield_X)} \big(
\pi^{-1} \C_{H_{\pm \fra}} \otimes \bigoplus_{c\in \Sigma} (\mathbb{E}^{c z} \otimes \Psi_c) \big) \simeq \mathfrak{t}
\end{equation}
where $\mathfrak{t} \subset \mathrm{End}_\C (\bigoplus_{c\in \Sigma} \Psi_c) $ is the subspace of block diagonal matrices. In particular, the isomorphisms \eqref{eq:salpha} are unique up to base-change by block-diagonal matrices in $\mathfrak{t} $.

Let $\mathbb{V} \defeq \bigoplus_{c\in \Sigma} \Psi_c $ and let us denote by $\mathrm{End}^{\pm}(\mathbb{V}) $ the subspace of upper/lower block triangular matrices with respect to the ordering \eqref{eq:orderingfrb}. We have the isomorphisms
\[
e_{\pm \frb}: 
\mathrm{End}_{\enh^{\mathrm{b}}_\Rc(\iCfield_X)} \big(
\pi^{-1} \C_{h_{\pm \frb}} \otimes \bigoplus_{c\in \Sigma} (\mathbb{E}^{cz} \otimes \Psi_c ) \big) \isoto
\mathrm{End}^{\pm}(\mathbb{V}) .
\]
\begin{definition}
 The \emph{Stokes matrices} of $\Hypabl $ at $z=\infty $ are the linear maps
 \[
 S_{\pm} \defeq e_{\pm \frb}^{-1}(\widetilde{s}_{\pm \frb}) \in \mathrm{End}^{\pm}(\mathbb{V}).
 \]
 Choosing a basis for $\mathbb{V} $ respecting the direct sum decomposition yields matrices $S_{\pm} \in \C^{n \times n} $ with $S_+ $ being upper and $S_- $ lower block triangular.
 \end{definition}

\begin{remark}
In \eqref{eq:Stokesdef} both Stokes matrices are defined as the transition from the solutions over $H_\alpha $ to $H_{-\alpha} $ -- see Figure \ref{fig:sectors}. The topological monodromy is the product $T^{\text{top}}= S_+^{-1} \cdot S_- $.
\end{remark}

\subsection{Ambiguity of Stokes matrices} \label{sec:ambiguity}

We keep the situation from above.
\begin{definition}
We fix the formal model $\model{M} := \bigoplus_{c\in \Sigma} ( E^{c z} \otimes R_c ) $ and let the regular singular connetions be determined by the formal monodromy $R_c  \simeq (\Psi_c, T_c) $. Then $ \Md{\model{M}} $ denotes the set of local isomorphism classes of germs of meromorphic connections $M $ at $z=\infty $ with
\[
\formalz{M} \cong \formalz{\model{M}}.
\]
\end{definition}

Any $[M]_\simeq \in \Md{\model{M}} $ gives rise to its pair of Stokes matrices $S_{\pm} $. We choose basis for $\Psi_j $ and read them as matrices in $\GL_n(\C) $. Let $r_c \defeq \dim(\Psi_c) $. The pair $\fra, \frb $ induces a total ordering on $\Sigma $ by \eqref{eq:orderingfrb}. Let us write 
\[
\Sigma= \{ c_1 <_\frb c_2 <_\frb \ldots <_\frb c_q \} \ .
\]
We decompose $\C^n=\bigoplus_{c\in \Sigma} \Psi_c $ and read the matrices $S \in \GL_n(\C) $ as composed into blocks according to this decomposition. We denote by $S^{ij} \in M(r_j \times r_i,\C) $ the corresponding block $S^{ij}:\Psi_{c_i} \to \Psi_{c_j} $, $i,j \in \{1, \ldots, q \} $.
Let 
\[
\Delta_{\Sigma} \defeq \GL_{r_1}(\C) \times \ldots \times \GL_{r_q}(\C)
\]
be the invertible \emph{block diagonal} matrices, 
\begin{align}\label{eq:defUL}
U_\Sigma & \defeq \{ S \in \GL_n(\C) \mid S^{ij}=0 \text{ for } c_i<_\frb c_j \} \\ \notag
L_\Sigma & \defeq  \{ S \in \GL_n(\C) \mid S^{ij}=0 \text{ for } c_i>_\frb c_j \} 
\end{align}
denote the \emph{block upper/lower triangular} matrices.

\begin{definition}\label{def:equivrel}
We define the following equivalence relation on $U_{\Sigma} \times L_{\Sigma} $
to be
\[
(S_+,S_-) \sim (S_+',S_-') :\Longleftrightarrow
\left\{ 
\begin{array}{l}
S_+'=AS_+ B \text{ and } \\
S_-'=AS_-B 
\end{array}\right\}
\text{ for some } A,B \in \Delta_{\Sigma}.
\]
Let $\Stmatr{\Sigma} $ denote the set of equivalence classes
\[
\Stmatr{\Sigma} := U_{\Sigma} \times L_{\Sigma} / \sim
 \]
and let us write $[S_+,S_-] \in \Stmatr{\Sigma} $ for the class represented by $(S_+,S_-) $.
\end{definition}

Clearly, the equivalence relation corresponds to the ambiguity in the choice of trivializations on the sectors as in \eqref{eq:salpha} (due to \eqref{eq:t}) and in the choice of a basis for $\Psi_c $, $c \in \Sigma $. The local classification thus can be stated as follows.
\begin{proposition}\label{prop:StokesRH}
In the situation above (including fixing $(\fra,\frb) $), the map
 \[
 \Md{\model{M}} \to \Stmatr{\Sigma} \ , \ M \mapsto [S_+, S_-]
 \]
is a bijection.
\end{proposition}

\section{Quiver associated to the perverse sheaf of a regular singular system}

\subsection{The quiver of a perverse sheaf}

We quickly recall the description of a perverse sheaves on the affine line with singular support in $\Sigma \subset \Af^1 $ in terms of quivers as in \cite[\S 4]{DHMS}. The description depends on a fixed pair $\fra,\frb \in \C^\times $ as in \eqref{eq:alphbet}. 

For each $c \in \Sigma $ we define $\ell_c:=c+ \R_{\ge 0} \fra $ and $\ell^\times_c:=\ell_c \smallsetminus \{ c \} $ to be the closed/open half lines starting at $c $ with direction $\fra $. Additionally, define $\ell_\Sigma:= \bigcup_{c\in \Sigma} \ell_c $. Recall that we have the total order $<_\frb $ on the set $\Sigma $ defined in \eqref{eq:orderingfrb}. For any $c \in \Sigma $, we write $c_- = \max \{ c' \in \Sigma \mid c'<_\frb c \} $ and $c_+ =\min \{ c' \in \Sigma \mid c<_\frb c' \} $. For $c=c_1 $, we use the notation $(c_1)_- = -\infty_\frb $, and for $c=c_\ell $ analogously $(c_\ell)_+=\infty_\frb $ and we agree that $\Re(\rglambda (-\infty_\frb)) < \Re(\rglambda z) < \Re(\rglambda \infty_\frb)  $ is true for all $z \in \Af^1 $. As in \cite{DHMS}, consider the strips
\begin{align*}
  B_c^> &\defeq \{ z \mid \Re(\frb c) < \Re(\frb z) <\Re(\frb c_+) \} &   B_c^> &\defeq \{ z \mid \Re(\frb c_-) < \Re(\frb z)<\Re(\frb c) \}  \\
 B_c^\ge &\defeq \{ z \mid\Re(\frb c) \le \Re(\frb z) <\Re(\frb c_+)\} &  B_c^\le &\defeq \{ z \mid \Re(\frb c_-) < \Re(\frb z) \le \Re(\frb c_+)\}\\
 B_c &\defeq \{ z \mid \Re(\frb c_-)< \Re(\frb z) < \Re(\frb c_+)\} 
\end{align*}
%\begin{align*}
%  B_0^> &\defeq \{ z \mid 0<\Re(\frb z)<\Re(\frb \rglambda) \} &   B_\rglambda^> &\defeq \{ z \mid \Re(\frb \rglambda)<\Re(\frb z) \}  \\
% B_0^\le &\defeq \{ z \mid \Re(\frb z) \le 0\} &  B_\rglambda^\le &\defeq \{ z \mid 0 < \Re(\frb z) \le \Re(\frb\rglambda) \}\\
% B_0 &\defeq \{ z \mid \Re(\frb z)< \Re(\frb \rglambda)\} &  B_\rglambda &\defeq \{ z \mid 0 < \Re(\frb z) \}.
%\end{align*}
The pair $\fra, \frb $ induces an orientation on $\C $ given by the basis consisting of $\fra $ and a tangent vector of the rotation of $\ell_c $ in the direction of $B_c^\le $. We will write $\field\defeq \C $.

\begin{figure}\centering
\begin{tikzpicture}[scale=.6,
	background rectangle/.style={draw=black,dashed,fill=white}, 
	show background rectangle]
\begin{scope}
\clip (-1.5,-1.2) rectangle (1.8, 1.5); 
\coordinate (C1a) at (.1,.3); % was (-.5,.3)
\coordinate (C1b) at (-.8,.3);
\coordinate (C3) at (1,.3);
\filldraw[fill=black, draw=black] (C1a) circle (2pt) node[below]{$c$};
\draw (C3) \crosss node[below]{$c_+$};
\draw (C1b) \crosss node[below]{$c_-$};
\draw[thick] (C1a) -- node[right]{$\ell_c$}++(0,1.3);
\end{scope}
\end{tikzpicture}
\qquad
\begin{tikzpicture}[scale=.6,
	background rectangle/.style={draw=black,dashed,fill=white}, 
	show background rectangle]
\begin{scope}
\clip (-1.5,-1.2) rectangle (1.8, 1.5); 
\coordinate (C1a) at (.1,.3);
\coordinate (C1b) at (-.8,.3);
\coordinate (C3) at (1,.3);
\filldraw[fill=black!20, draw=black, style=densely dotted] (.1,-2) rectangle (1,2);
\filldraw[fill=white, draw=white] (C1a)++(-.02,0) rectangle (-.48,2);
\draw[densely dotted] (C1a) -- node[right]{$B_c^>$}++(0,1.3);
\filldraw[fill=white, draw=black] (C1a) circle (2pt) node[below]{$c$};
\draw (C3) \crosss node[below]{$c_+$};
\draw (C1b) \crosss node[below]{$c_-$};
\end{scope}
\end{tikzpicture}
\qquad
\begin{tikzpicture}[scale=.6,
	background rectangle/.style={draw=black,dashed,fill=white}, 
	show background rectangle]
\begin{scope}
\clip (-1.5,-1.2) rectangle (1.8, 1.5); 
\coordinate (C1a) at (.1,.3);
\coordinate (C1b) at (-.8,.3);
\coordinate (C3) at (1,.3);
\filldraw[fill=black!20, draw=black, style=densely dotted] (.1,-2) rectangle (1,2);
\filldraw[fill=white, draw=white] (C1a)++(-.02,0) rectangle (-.48,2);
\draw[densely dotted] (C1a) -- node[right]{$B_c^\ge$}++(0,1.3);
\filldraw[fill=black, draw=black] (C1a) circle (2pt) node[below,left]{$c$};
\draw (C3) \crosss node[below]{$c_+$};
\draw (C1b) \crosss node[below]{$c_-$};
\draw[thick] (C1a) -- ++(0,1.3);
\draw[thick] (C1a) -- ++(0,-1.5);
\end{scope}
\end{tikzpicture}
\qquad
\begin{tikzpicture}[scale=.6,
	background rectangle/.style={draw=black,dashed,fill=white}, 
	show background rectangle]
\begin{scope}
\clip (-1.5,-1.2) rectangle (1.8, 1.5); 
\coordinate (C1a) at (.1,.3);
\coordinate (C1b) at (-.8,.3);
\filldraw[fill=black!20, draw=black, style=densely dotted] (-2,-2) rectangle (2,2);
\filldraw[fill=white, draw=white] (C1a)++(-.02,0) rectangle (.12,2); %(-.48,2);
\draw[densely dotted] (C1a) --  ++(0,1.3);
\filldraw[fill=white, draw=black] (C1a) circle (2pt);
\filldraw[fill=white, draw=white] (C1b)++(-.02,0) rectangle (-.78,2); %(-.48,2);
\draw[densely dotted] (C1b) --  ++(0,1.3);
\filldraw[fill=white, draw=black] (C1b) circle (2pt);
\filldraw[fill=white, draw=white] (C3)++(-.02,0) rectangle (1.02,2);
\draw[densely dotted] (C3) --  ++(0,2);
\filldraw[fill=white, draw=black] (C3) circle (2pt);
\draw (.5,-.4) node {$\mathbb{A}^1 \smallsetminus \ell_\Sigma $};
\end{scope}
\end{tikzpicture}
\caption{The sets $\ell_{c}$, $B_{c}^>$, $B_c^\ge $, and $\Af^1 \smallsetminus \ell_\Sigma$.}\label{fig:Quiver1}
\end{figure}

\begin{definition}[{\cite[Definition 4.6]{DHMS}}]
Let $c\in \Sigma$ and $F \in\Perv_\Sigma(\Af^1)$. The complexes of \emph{nearby cycles at $c$}, \emph{vanishing cycles at $c$}, and \emph{global nearby cycles} are defined by the formulas
\begin{align*}
\Psi_c^{(\fra,\frb)}(F)=\Psi_c(F) &\defeq \rsect_\rc(\Af^1; \field_{\ell^\times_c} \otimes F), \\
\Phi_c^{(\fra,\frb)}(F)=\Phi_c(F) &\defeq \rsect_\rc(\Af^1;\field_{\ell_c} \otimes F),\\
\Psi^{(\fra,\frb)}(F)=\Psi(F) &\defeq \rsect_\rc(\Af^1;\field_{\Af^1\setminus\ell_\Sigma} \otimes F)[1].
\end{align*}
\end{definition}
Note that these definitions do not depend on $\frb $, but the following construction of the canonical and variation maps does. If $L:=F|_U $ denotes the associated local system on $U:=\Gm \smallsetminus \{1  \} $, it is easily seen that
\begin{align*}
 \Psi_c(F) & \simeq H^0 \Psi_c(F) \simeq H^0 \rsect_\rc(\ell^\times_c,F) \simeq H^1_c(\ell^\times_c,L) \\
 \Phi_c(F) & \simeq H^0 \Phi_c(F) \simeq H^0 \rsect_\rc(\ell_c,F) \\
  \Psi(F) & \simeq H^0 \Psi_\Sigma(F)  \simeq H^2_c(\Af^1 \smallsetminus \ell_\Sigma, L).
\end{align*}

Let us fix an isomorphism between the local and the global nearby cycles. To this end consider the short exact sequences (using the notation $\field_Z := (j_Z)_! (j_Z)^\ast \underline{\field} $ for a subspace $j_Z: Z \hookrightarrow \Af^1 $):
\begin{align*}
 & 0 \to \field_{B_c^>} \to \field_{B_c^> \cup \ell^\times_c} \to \field_{\ell^\times_c} \to 0 \\
 & 0 \to \field_{B_c^>} \to \field_{\Af^1 \smallsetminus \ell_\Sigma} \to \field_{(\Af^1 \smallsetminus \ell_\Sigma) \smallsetminus B_c^>} \to 0 
\end{align*}
The spaces $B_c^> \cup \ell^\times_c $ and $(\Af^1 \smallsetminus \ell_\Sigma) \smallsetminus B_c^> $ produce no cohomology with compact support for any $F \in \Perv_\Sigma $ -- in the notion of \cite{DHMS}, they are $\Sigma $-negligible. Therefore, we obtain the isomorphisms
\begin{equation}\label{eq:isolocglob}
\psi_c:\Psi_c(F) \isoto \rsect_\rc(\Af^1; \field_{B^>_c} \otimes F)[1] \isoto \rsect_\rc(\Af^1; \field_{\Af^1 \smallsetminus \ell_\Sigma} \otimes F)[1] = \Psi(F) .
\end{equation}

Now, the local version of the canonical morphism
\[
u_{cc}:\Psi_c(F) = \rsect_\rc(\Af^1; \field_{\ell^\times_c} \otimes F) \to 
\rsect_\rc(\Af^1; \field_{\ell_c} \otimes F) = \Phi_c(F)
\]
is induced by the canonical inclusion. It defines the global canonical map
\begin{equation}
 u_c: \Psi(F) \isoto[\psi_c^{-1}] \Psi_c(F) \to[u_{cc}] \Phi_c(F) 
\end{equation}

The exact squences
\begin{align*}
& 0 \to \field_{B_c \smallsetminus \ell_c} \to \field_{B_c} \to \field_{\ell_c} \to 0 \\
& 0 \to \field_{B_c \smallsetminus \ell_c} \to \field_{\Af^1\smallsetminus \ell_\Sigma} \to \field_{(\Af^1\smallsetminus \ell_\Sigma) \smallsetminus B_c} \to 0
\end{align*}
give rise to the morphism and isomorphism (since $(\Af^1\smallsetminus \ell_\Sigma) \smallsetminus B_c $ is $\Sigma $-negligible)
\begin{multline*}
v_c: \Phi_c(F)=\rsect_\rc(\Af^1; \field_{\ell_c} \otimes F) \to \rsect_\rc(\Af^1; \field_{B_c \smallsetminus \ell_c} \otimes F)[1] \isoto \\
\rsect_\rc(\Af^1; \field_{\Af^1 \smallsetminus \ell_\Sigma} \otimes F)[1] = \Psi(F)
\end{multline*}
As with the map $u_c $, there is a local variant
\[
v_{cc}= (\psi_c)^{-1} \circ v_c: \Psi_c(F) \to \Psi(F) \to \Psi_c(F) \ .
\]
For any $F \in \Perv_\Sigma(\Af^1) $, the \emph{quiver} of $F $ is defined to be
\[
Q^{(\fra,\frb)}_\Sigma(F):= 
\xymatrix{
{\big( \Psi(F)} \ar@<.5ex>[r]^{u_\rglambda}  & 
{\Phi_\rglambda(F) \big)_{c \in \Sigma}} \ar@<.5ex>[l]^{v_\rglambda}}
\]
This defines the functor
\[
Q^{(\fra,\frb)}_\Sigma(F): \Perv_\Sigma(\Af^1) \to \Quiv_\Sigma
\]
to the abelian category of quivers indexed by $\Sigma $, i.e. diagrams of finite dimensional vector spaces and linear maps of the form
\[
\xymatrix{
{\big( W} \ar@<.5ex>[r]^{u_\rglambda}  & 
{V_\rglambda \big)_{c\in \Sigma}} \ar@<.5ex>[l]^{v_\rglambda}}
\]
such that $T_c:=1-v_c \comp u_c $ is an isomorphism of $V_c $ for $c\in \Sigma $. The morphisms in this category are linear maps compatible with the given linear maps $u_c $, $v_c $. This functor is an equivalence (see \cite[Corollary 4.17]{DHMS}). The following is easy to see. 
\begin{lemma}\label{lem:quivker} 
Given a morphism
 \[
 \xymatrix{
{\big( W} \ar@<.5ex>[r]^{u_c}  \ar[d]_{f} & 
{V_c\big)_{c\in \Sigma}} \ar@<.5ex>[l]^{v_c} \ar@<-2ex>[d]^{f_c} \\
{\big(W'}  \ar@<.5ex>[r]^{u'_\rglambda}  & 
{V'_c\big)_{c\in \Sigma}} \ar@<.5ex>[l]^{v'_c} }
 \]
 \begin{enumerate}
 \item\label{ker} its kernel is the quiver
 $
\xymatrix{
{\big(\ker(f)}  \ar@<.5ex>[r]^{u_c}  & 
{\ker(f_c)\big)_{c\in \Sigma}} \ar@<.5ex>[l]^{v_c}}
 $ 
with the corresponding restrictions of $u_c $, $v_c $, and 
\item\label{coker} its cokernel is the quiver
 $
\xymatrix{
{\big(\coker(f)} \ar@<.5ex>[r]^{u'_c}  & 
{\coker(f_c)\big)_{c\in \Sigma}} \ar@<.5ex>[l]^{v'_c}}
$
with the naturally induced maps $u'_c $, $v'_c $.
\end{enumerate}
\end{lemma}

\begin{lemma}
 The perverse sheaf $F $ has support inside $\Sigma $ if and only if its quiver is of the form
 $
 Q^{(\fra,\frb)}_\Sigma(F)= 
 \xymatrix{
{\big(0} \ar@<.5ex>[r]  & {V_c\big)_{c\in \Sigma}} \ar@<.5ex>[l]}.
 $
\end{lemma}
\begin{proof}
The support of $F $ lies inside $\Sigma $ if and only if the local system $L:=F|_U $ vanishes and equivalently, its nearby cycles vanish.
\end{proof}

Let us denote by $i:\Af^1\smallsetminus \Sigma \hookrightarrow \Af^1 $ the inclusion.
\begin{corollary}\label{cor:middleperv}
 The perverse sheaf $F $ is its own middle extension $F \simeq i_{!\ast} i^\ast F=i_{!\ast} L[1] $ if and only if its quiver is isomorphic to a quiver of the form
 \[
 \xymatrix{
{\big( \Psi(F)} \ar@<.5ex>[r]^(.4){1-T_c}  & 
{\im(1-T_c) \big)_{c\in \Sigma}} \ar@<.5ex>[l]^(.6){\iota_c}}
 \]
 with the monodromy operators $T_c:=1-v_c \comp u_c $ and the canonical inclusions $\iota_c $.
\end{corollary}
\begin{proof}
 The middle extension $i_{!\ast} L[1] $ is characterized by the property that it does neither assume a non-trivial subobject nor quotient supported on $\Sigma $.
 
A subobject of the quiver $Q(F) $ of $F $ supported on $\Sigma $ is of the form
\[
\xymatrix{
{\big( 0}  \ar@<.5ex>[r]  \ar[d]_{f} & 
{V_c \big)_{c\in \Sigma}} \ar@<.5ex>[l] \ar@<-2ex>[d]^{f_c} \\
{\big(\Psi(F)} \ar@<.5ex>[r]^{u_c}  & 
{\Phi_c(F) \big)_{c\in \Sigma}} \ar@<.5ex>[l]^{v_\rglambda} }
\]
with injective $f_c $ (Lemma \ref{lem:quivker} \eqref{ker}). In particular, $f_c(V_c) \subset \ker(v_c) $. In the other direction, any non-trivial subspace of some $\ker(v_c) $ defines such a subobject. Therefore, we see that $Q(F) $ does not admit a non-trivial subobject if and only if all maps $v_c $ are injective.

A quotient of $Q(F) $ supported on $\Sigma $ is of the form
\[
\xymatrix{
{\big( \Psi(F)}  \ar@<.5ex>[r]^{u_c}  \ar[d]_{f} & 
{\Phi_c(F)\big)_{c\in \Sigma}} \ar@<.5ex>[l]^{v_c} \ar@<-2ex>[d]^{f_c} \\
{\big( 0} \ar@<.5ex>[r]  & {V_\rglambda\big)_{c\in \Sigma}} \ar@<.5ex>[l] }
\]
with surjective $f_c $ (Lemma \ref{lem:quivker} \eqref{coker}). In particular, for any $x \in \Psi(F) $, we know that $f_c \comp u_c(x)=0 $, i.e. $\im(u_c) \subset \ker(f_c) $. Therefore, if $u_c $ is surjective, we deduce from the surjectivity of $f_c $ that $V_c=0 $. 

On the other hand side, if one of the $u_c $ is not surjective, we can choose $V_c:=\Phi_c(F)/\im(u_c) $ and get a non-trivial quotient object. Therefore, $Q(F) $ does not admit any non-trivial quotient supported on $\Sigma $ if and only if $u_0 $ and $u_\rglambda $ are surjective.

If we now assume, that $v_c $ are injective and $u_c $ surjective, we have the isomorphism of quivers
\[
 \xymatrix{
{\big(\Psi(F)} \ar@<.5ex>[r]^{u_c}  \ar[d]_{\mathrm{id}}^{\simeq} & 
{\Phi_c(F)\big)_{c\in \Sigma}} \ar@<.5ex>[l]^{v_c} \ar@<-2ex>[d]_{v_c}^{\simeq} \\
 {\big(\Psi(F)} \ar@<.5ex>[r]^(.4){1-T_c}  & 
{\im(1-T_c) \big)_{c\in \Sigma}} \ar@<.5ex>[l]^(.6){\iota_c}}
\]
keeping in mind that $\im(v_c)=\im(v_c \comp u_c)=\im(1-T_c) $.
\end{proof}

\subsection{Application to a regular singular hypergeometric system}

Let $\ugamma=(\gamma_1, \ldots, \gamma_n) $ and $\ueta=(\eta_1, \ldots, \eta_n) $ be a non-resonant pair of parameters. The solutions of the regular singular module $\Hypl{\ugamma}{\ueta}{\rglambda} $ give a perverse sheaf on $\Gm $ and we define
\[
F:= \jint \mathcal{S}ol( \Hypl{\ugamma}{\ueta}{\rglambda} ) \in \Perv_\Sigma(\Af^1) \ ,
\]
a perverse sheaf on $\Af^1 $ with singularities in $\Sigma= \{0, \rglambda \} $. 

\begin{convention}
We fix $\fra \defeq i \cdot \rglambda $ and $\frb \defeq \rglambda^{-1}$. 
\end{convention}
Due to the irreducibility, we know that $F $ is its own middle extension at $\rglambda $. We deduce the following result
\begin{proposition}\label{prop:quiverF}
The quiver of the perverse sheaf $F:= \jint \mathcal{S}ol( \Hypl{\ugamma}{\ueta}{\rglambda} ) $ for non-resonant $\ugamma, \ueta \in \C^n $ is isomorphic to 
\[
Q(F) \simeq \big( 
\xymatrix{
{\im(1-T_0)} \ar@<-.5ex>[r]_(.6){\iota_0} & {\C^n} \ar@<-.5ex>[l]_(.4){1-T_0} \ar@<.5ex>[r]^(.4){1-T_\rglambda}  & 
{\im(1-T_\rglambda)} \ar@<.5ex>[l]^(.6){\iota_\rglambda}} \big)
\]
where $T_c $ is the monodromy of the local system of solutions around $c $ and $\iota_c $ are the inclusions.
\end{proposition}

\subsection{Monodromy of the regular singular hypergeometric system} \label{subsec:mdrmy}

We now recall the results of Beukers-Heckman on the (global) monodoromy of an irreducible regular singular hypergeometric system.

We pick up the notation from the last section. In particular, we denote by $F $ the perverse sheaf on $\Af^1 $ given by the solutions of $\jint \Hypl{\ugamma}{\ueta}{\rglambda} $ for some non-resonant parameters $\ugamma, \ueta \in \C^n $. Let us take a closer look at the (global) monodromy operators
\[
T_c \defeq 1- v_c \comp u_c \in \text{Aut}(\Psi(F)) 
\]
for $c \in \Sigma=\{0, \rglambda \}$ recalling results from \cite[4.2]{DHMS}. To this end, let us fix a base-point $b_c $ in $\ell^\times_c $ and denote by $(c,b_c) $ or $(c,b_c] $ the corresponding part of $\ell_c^\times $. The choice of the pair $(\alpha,\beta) $ fixes an orientation on $\Af^1 $ (here it is counter-clockwise) and hence an isomorphism $\pi_1(B_c \smallsetminus \{ c \}, b_c) \simeq \Z $. We let $L $ denote the associated local system such that $F|_{\Af^1 \smallsetminus \Sigma}=L[1] $. We then have the local topological monodromy operator
\[
T_{cc}^{\text{top}} \in \text{Aut}(L_{b_c})
\]
induced by the positive generator of $\pi_1(B_c \smallsetminus \{ c \}, b_c) \simeq \Z $.

Let $i_{b_c}:\{ b_c\} \hookrightarrow \Af^1 $ be the inclusion. We have a natural isomorphism $\chi_c:i_{b_c}^{-1}F[-1] \isoto \Psi_c(F) $ via the inclusion $(c,b_c) \to \ell_c $ and the short exact sequence
\[
0 \to \field_{(c,b_c)} \to \field_{(c,b_c]} \to \field_{\{ b_c \}} \to 0 .
\]
We have $i_{b_c}^{-1}F[-1] \simeq L_{b_c} $ and this isomorphism intertwines the local topological monodromy $T_{cc}^{\text{top}} $ and $T_{cc}=1-v_{cc} \circ u_{cc} $, see Lemma 4.10 in \cite{DHMS}. In summary, we get
\[
T_c = (\psi_c \circ \chi_c) \circ T_{cc}^{\text{top}} \circ (\psi_c \circ \chi_c)^{-1} 
\]
for $c \in \Sigma=\{0,\rglambda \} $. Let us define
\[
T_\infty := T_0^{-1} \circ T_\rglambda^{-1} \in \text{Aut}(\Psi(F)) 
\]
so that $T_\infty T_\rglambda T_0=1 $. Up to conjugation, this operator corresponds to the local monodromy around $\infty $. 

The fundamental group $\pi_1(\Pj^1 \smallsetminus \{0,\rglambda,\infty\}, b) $ is generated by the homotopy classes $g_c $ of simple loops (counter-clockwise) around the singularities $c=0,\rglambda,\infty $ starting at a base-point $b $ as in the Figure \ref{fig:fundgroup}. We choose the loops such that
\[
g_0 g_\rglambda g_\infty = 1 \in \pi_1(\Pj^1 \smallsetminus \{0,\rglambda,\infty\},b) \ .
\]

\begin{figure}\centering
\begin{tikzpicture}[scale=.7,
	background rectangle/.style={draw=black,dashed,fill=white}, 
	show background rectangle]
\coordinate (C1) at (2,1);
\coordinate (C2) at (.5,0); %vorher (0,
\coordinate (C3) at (-.5,0);  %vorher (.5,
\coordinate (C4) at (-1.3,0);  %vorher (.5,
\coordinate (C5) at (3,0);  %vorher (.5,

\begin{scope}[decoration={
    markings,
    mark=at position .7 with {\arrow{>}}}
    ]
\clip (-1.5,-1.6) rectangle (3.1, 1.8);

\draw[black] (C3) circle (2pt) node[below]{$0$} ;
\draw[black] (C1) circle (2pt) node[below]{$\rglambda$} ;
\filldraw[black] (C2) circle (2pt) node[right]{$b$} ;

\path[draw=black,-,postaction={decorate}] (C2) to[left, loop, out=180-45, in=180+45, min distance=30mm] (C2);
\draw[black] (C3) ++(.5,-.7) node []{$g_0$};

\path[draw=black,-,postaction={decorate}] (C2) to[left, loop, out=0, in=60, min distance=35mm] (C2);
\draw[black] (C1) ++(-.1,-.7) node [below]{$g_\rglambda$};

\draw [black,postaction={decorate}] (C2) 
  .. controls ++(70:2) and ++(90: 2) .. (C5)
 .. controls ++(90:-2) and ++(90:-2) .. (C4)
  .. controls ++(90:1) and ++(115:1) .. (C2);

\draw[black] (C2) ++(.5,-1.2) node[] {$g_\infty $};
 
 \end{scope}
\end{tikzpicture}
\caption{The generators of the fundamental group.}\label{fig:fundgroup}
\end{figure}

If we denote by $h_c $ the monodromy induced by analytic continuation along $g_c $ -- which is of course conjugate to $T_c $ above, the representation map $\pi_1(\Pj^1\smallsetminus \{ 0,\rglambda,\infty\}, b) \to \text{Aut}(\shl_{b}) $, $g_c \mapsto h_c $ is an \emph{anti-}homomorphism if we write the composition in the fundamental group in the ''standard'' way ($g_0g_\rglambda$ meaning $g_0 $ first, $g_\rglambda $ second) -- note that Beukers-Heckmann use the opposite group multiplication in order to have a group homomorphism instead of an anti-homomorphism. In the standard notation, we have the homotopy relation $g_0g_\rglambda g_\infty \simeq 1$, which leads to the monodromy relation $h_\infty \circ h_\rglambda \circ h_0 = \id $.

The eigenvalues of the monodromies around $c=0,\rglambda $ and $\infty $ of the hypergeometric equation $\opHypl{\ugamma}{\ueta}{\rglambda} $ (cp \cite[Corollary 3.2.2]{katz}) are
\begin{align}\label{eq:locmdrmy}
 \eig{T_0} &= \{\exp(2\pi i \gamma_i) \mid i=1, \ldots, n\} \\
\notag \eig{T_\infty} &= \{\exp(-2\pi i \eta_i) \mid i=1, \ldots, n\} \\ 
\notag \eig{T_\rglambda} &= \{ \underbrace{1, \ldots, 1}_{(n-1) \text{ times}} , \exp(2\pi i \lambda)\} \ .
\end{align}
%The last line is of particular interest, since it reflects the fact that there are $(n-1) $ linear independent holomorphic solutions of $\opHypl{\ugamma}{\ueta}{\rglambda} $ around $c=\rglambda $. This is usually referred to as a Theorem of Pochhammer. 

In particular, $T_\rglambda $ is a pseudo-reflection, i.e. $\rank(1-T_\rglambda)=1 $, for generic $\ugamma, \ueta $. Now, there is the following beautiful observation due to Levelt (cp. \cite[Theorem 3.5]{BH}). To facilitate its statement, let us introduce some notations first. For a given $a \in \C^n $, let $\prod_{i=1}^n (X-a_i) = X^n+A_1X^{n-1}+ \ldots +A_n $
denote the coefficients of the polynomials with the given roots. We will use the notation
\begin{equation}\label{eq:companiondef}
\companion{a} \defeq \companion{a_1, \ldots, a_n} \defeq 
\begin{pmatrix}
 0 &&&&& -A_n\\
 1 & 0 &&&& -A_{n-1}\\
 & \ddots &\ddots &&&\vdots \\
 &&&&0 &-A_2 \\
 &&&&1 & -A_1
\end{pmatrix}
\end{equation}
for the corresponding companion matrix. 

\begin{lemma}[Levelt] \label{lemma:Levelt}
Let $V $ be an $n $-dimensional complex vector space and $a,b \in (\C^\times)^n $ be given, such that $a_i \neq b_j $ for all $i,j $. If $(A,B) \in \GL(V) \times \GL(V) $ are given such that
\begin{align*}
& \eig{A}  =\{ a_1, \ldots, a_n \}, \\
 &\eig{B}  =\{b_1, \ldots, b_n \} , \\
&AB^{-1}  \text{ is a pseudo-reflection,}
\end{align*}
then there is a choice of basis $\varphi:V \isoto \C^n $ such that
\begin{align*}
\varphi \circ A \circ \varphi^{-1}  &= \companion{a_1, \ldots, a_n} \in \GL_n(\C) \text{ and} \\
\varphi \circ B \circ \varphi^{-1} &= \companion{b_1,\ldots, b_n} \in \GL_n(\C) \ .
\end{align*}
\end{lemma}
\begin{proof} 
see \cite[Theorem 3.5]{BH}.
\end{proof}
%For readability reasons, we repeat the proof from \cite{BH}. 
%
%Note first, that due to the assumptions, $\companion{a} \circ \companion{b}^{-1} $ is a pseudo-reflection since
%\[
%\rank(1-\companion{a}\companion{b}^{-1})=\rank((\companion{b}-\companion{a})=1 .
%\]
%Now, consider the $(n-1) $-dimensional subspace $W\defeq \ker(A-B) $. Define
%\[
%U:= W \cap A^{-1} W \cap \ldots \cap A^{-(n-2)} W \ ,
%\]
%with $\dim(U) \ge 1 $. Assuming $\dim(U)>1 $, we find a vector $0 \neq v \in U \cap A^{-(n-1)}W $ and deduce that $A^i v \in W $ for all $i=0, 1, \ldots, n-1 $. Therefore $W $ contains the $A $-stable subspace
%\[
%0 \neq \langle A^i v \mid i \in \Z \rangle \subset W 
%\]
%and in particular, there is an eigenvector of $A $ inside $W $. Since $A|_W=B|_W $, this is also an eigenvector of $B $ with the same eigenvalue in contradiction to the assumptions on $a,b $. 
%
%Hence, $\dim(U)=1 $, say $U= \C \cdot v $. Then $A^i v=B^i v $ for all $i=0, \ldots, n-2 $ and if we take $v, Av, \ldots, A^{n-1}v $ as our basis for $V $, the linear maps $A, B $ are represented in this basis by the companion matrices as desired.
%\end{proof}
 
We apply this Lemma to $A \defeq T_\infty^{-1} $ and $B \defeq  T_0 $. Then $T_\rglambda=T_\infty^{-1} \circ T_0^{-1}=AB^{-1} $ is a pseudo-reflection.

\begin{corollary}\label{cor:mdrmy}
 Let $\ugamma,\ueta $ be non-resonant and $F $ be the perverse sheaf $F\defeq \jint \mathcal{S}ol( \Hypl{\ugamma}{\ueta}{\rglambda} ) \in \Perv_\Sigma(\Af^1) $. Then there is an isomorphism $\varphi: \Psi(F) \isoto \C^n $ such that
\begin{align*}
 \varphi \circ T_0 \circ \varphi^{-1} & = \companion{\exp(2\pi i \ugamma)}
 \\
 \varphi \circ T_\rglambda \circ \varphi^{-1} & =  \companion{\exp(2\pi i \ueta)} \circ \companion{\exp(2\pi i \ugamma)}^{-1}
\end{align*} 
\end{corollary}

Let us introduce the following notation for a given $\ugamma $:
\begin{equation}\label{eq:Cgamma}
\Co{\ugamma} := \companion{\exp(2\pi i \ugamma)} \ .
\end{equation}
If we denote by
\begin{align}\label{eq:defchiC}
\chi_C(X) \defeq \prod_{i=1}^n (X-\exp(2\pi i\gamma_i)) &= X^n + C_1X^{n-1} + \ldots + C_n \\
\chi_E(X) \defeq \prod_{i=1}^n (X-\exp(2\pi i\eta_i)) &= X^n+ E_1X^{n-1}+ \ldots + E_n \label{eq:defchiE}
\end{align}
the coefficients of the polynomials, we get
\[
1-\Co{\ueta}\Co{\ugamma}^{-1}=
\left(\begin{array}{c|ccccc}
\begin{array}{c}
(C_{n}-E_{n})C_n^{-1} \\
%(C_{n-1}-E_{n-1})C_n^{-1} \\
%(C_{n-2}-E_{n-2})C_n^{-1} \\
\vdots \\
(C_1-E_1)C_n^{-1} 
\end{array}
&&& \mbox{$\text{\Large $0$}_{n\times (n-1)}$} &&
\end{array}\right)
\]

\subsection{Explicit presentation of the quiver}

Summing up the results of the previous subsections, we obtain the following.

\begin{proposition}\label{prop:quiverR}
For non-resonant $(\ugamma, \ueta) $,  the perverse sheaf of solutions of $\Hypl{\ugamma}{\ueta}{\rglambda} $ is (up to isomorphism of quivers) given by
\begin{equation}\label{eq:quivHyp1}
\xymatrix@!C{
 {\im(1-\Co{\ugamma})} \ar@<-.5ex>[r]_(.6){\iota_0}  & {\C^n} \ar@<-.5ex>[l]_(.4){1-\Co{\ugamma}} \ar@<.5ex>[r]^(.4){1-\Co{\ueta}\Co{\ugamma}^{-1}}  & 
{\im(1-\Co{\ueta}\Co{\ugamma}^{-1})} \ar@<.5ex>[l]^(.6){\iota_\rglambda}
}
\end{equation}
\end{proposition}

We want to give two explicit representatives of the isomorphism class of this quiver which we call the \emph{companion representative} and the \emph{Jordan representative} respectively.

\begin{assumption}\label{ass:1EW}
 We assume that $1 $ is an eigenvalue of $\Co{\ugamma} $, and we let $\gamma_n=1 $.  For the coefficients of the characteristic polynomial, this yields
\begin{equation}\label{eq:sumC}
 \sum_{i=1}^n C_i +1 = 0 \ .
\end{equation}
\end{assumption}
In the following, we will make this assumption (which will be satisfied in our application).

\subsubsection{The companion representative}

Since
\[
1-\Co{\ugamma} =
\begin{pmatrix}
 1 &&&&& C_n\\
 -1 & 1 &&&& C_{n-1}\\
 && \ddots &\ddots &&\vdots \\
 &&&&1 &C_2 \\
 &&&&-1 & 1+C_1
\end{pmatrix}
\]
we obtain under Assumption \ref{ass:1EW} that $\im(1-\Co{\ugamma}) = \{ x \in \C^n \mid \sum_{i=1}^n x_i=0 \} $.
We choose the basis $(e_1-e_2, e_2-e_3, \ldots, e_{n-1}-e_n) $ (with $e_i $ the standard basis vector of $\C^n $) and denote by $\vi_0: \C^{n-1} \isoto \im(1-\Co{\ugamma}) $ the corresponding isomorphism sending the standard basis to the latter. 

For the vanishing cycles at $1 $, we see that $\im(1-\Co{\ueta}\Co{\ugamma}^{-1})=\im(\Co{\gamma}-\Co{\ueta}) $ and we fix the isomorphism
\begin{equation} \label{eq:vi1}
\vi_\rglambda:\C \isoto \im(\Co{\ugamma}-\Co{\ueta}) \ , \ 1 \mapsto {}^t(E_n-C_n, E_{n-1}-C_{n-1}, \ldots, E_1-C_1) \ .
\end{equation}
Elementary computations give the following.
\begin{proposition}\label{ref:propcompanion}
 Under Assumption \ref{ass:1EW}, we have the following isomorphism of quivers -- the first line being the quiver of \eqref{eq:quivHyp1}:
\[
\xymatrix@!C{
 {\im(1-\Co{\ugamma})} \ar@<-.5ex>[r]_(.6){\iota_0}  & {\C^n} \ar@<-.5ex>[l]_(.4){1-\Co{\ugamma}} \ar@<.5ex>[r]^(.4){1-\Co{\ueta}\Co{\ugamma}^{-1}}  & 
{\im(\Co{\ugamma}-\Co{\ueta})} \ar@<.5ex>[l]^(.6){\iota_\rglambda} \\
{\C^{n-1}} \ar@<-.5ex>[r]_(.6){V'_0} \ar[u]^{\simeq}_{\vi_0} & {\C^n} \ar@<-.5ex>[l]_(.4){U_0'} \ar[u]^{\simeq}_{\mathrm{id}} \ar@<.5ex>[r]^(.4){U_\rglambda'}  & 
{\C} \ar@<.5ex>[l]^(.6){V'_\rglambda} \ar[u]^{\simeq}_{\vi_\rglambda}
} 
\]
with
\begin{align*}
U_0'  & \defeq 
\left(\begin{array}{c|l}
1_{n-1} &
\begin{array}{l}
C_n \\
C_n+C_{n-1}\\
\vdots \\
C_n+C_{n-1}+ \ldots + C_2
\end{array}
\end{array}\right) \ &, \ 
V_0' & \defeq 
\begin{pmatrix}
 1\\
 -1&1\\
 &\ddots & \ddots\\
 &&-1 &1\\
&&&-1
\end{pmatrix} \ , \\
U_\rglambda' & \defeq ( -C_n^{-1}, 0, \ldots, 0) \ &, \ V_\rglambda' &\defeq {}^t(E_n-C_n, \ldots, E_1-C_1) .
\end{align*}
\end{proposition}

\begin{remark}\label{rema:comp}
 \begin{enumerate}
 \item By Assumption \ref{ass:1EW}, we have $C_1=-1-(C_n+\ldots+C_2) $, so no information is lost in $U_0' $.
\item  \label{rema:diag}
In order to better understand the entries of $U_0' $, let us observe that under Assumption \ref{ass:1EW}, we have
\begin{multline*}
X^n+C_1X^{n-1} + \ldots + C_{n-1}X+ C_n =\\
 (X-1) \cdot \big( X^{n-1}-(C_n+\ldots+C_2) X^{n-2} - (C_n+ \ldots+C_3)X^{n-3} + \ldots - C_n \big) 
\end{multline*}
hence the second factor equals $\prod_{j=1}^{n-1} (X-\exp(2 \pi i \gamma_j)) $ and we have $1- U_0'V_0' = 
%\left(
%\begin{array}{ccccl}
%0 & &&& C_n \\
%1 &0 &&& C_n+C_{n-1} \\
%&\ddots &&& \vdots \\
%&&&1& C_n+ \ldots + C_2
%\end{array}\right) =
 \Co{\ugamma'} \in \GL_{n-1}(\C)
$ where we write $\ugamma=(\ugamma',1) $. 
\end{enumerate}
\end{remark}

\subsubsection{The Jordan representative}\label{sec:Jordanquiver}

We want to write the quiver in terms of the Jordan normal form of $\Co{\ugamma} $. Note that the parameters $\gamma_1, \ldots, \gamma_{n-1},1 $ are not necessarily pairwise distinct. We will write $\{\lambda_1, \ldots,  \lambda_\ellJ \} $ for the pariwise distinct eigenvalues with $\lambda_\ellJ=1 $ and order the exponents $\ugamma $ accordingly:
\begin{align}\label{eq:lambdamult}
\lambda_1 & \defeq \exp(2\pi i\gamma_1) = \ldots = \exp(2\pi i\gamma_{k_1}) \\ \notag
\lambda_2 & \defeq \exp(2\pi i\gamma_{k_1+1}) = \ldots = \exp(2\pi i\gamma_{k_1+k_2}) \\ \notag
\vdots & \hspace*{2cm}\vdots \\ \notag
1=\lambda_\ellJ & \defeq \exp(2\pi i\gamma_{n-k_\ellJ}) = \ldots = \exp(2\pi i\gamma_n) .
\end{align}
Hence, $k_j $ is the algebraic multiplicity of the eigenvalue $\lambda_j $.

\begin{remark}
For each individual eigenvalue $\lambda_j $, the eigenspace of the companion matrix is one-dimensional: $\dim\ker(\Co{\ugamma} - \lambda_j \cdot \id) = 1 $. We will write $\Jblock{\lambda_j} $ for the corresponding Jordan block
\[
\Jblock{\lambda_j} \defeq 
\begin{pmatrix}
 \lambda_j & 1  \\
% & \lambda_j & 1\\
 & \ddots & \ddots \\
 && \lambda_j & 1 \\
 &&&\lambda_j
\end{pmatrix} \in \mathrm{End}(\C^{k_j})
\]
of size $k_j \times k_j $.
\end{remark}

%Due to Remark \ref{rema:comp} (\ref{rema:diag}), we know that, $\ugamma' $ being generic, the map $1-U_0'V_0' \in \GL_{n-1}(\C) $ is diagonalizable with eigenvalues $\exp(2\pi i \gamma_1), \ldots, \exp(2 \pi i \gamma_{n-1}) $. Since usually in the literature, the Stokes matrices are given with diagonal blocks in the diagonal, we want to give another representative of the isomorphism class of the quiver \eqref{eq:quivHyp1} leading to such a form of Stokes matrices in the following section.

Let $H \in \GL_n(\C) $ be a base change such that
\begin{equation}\label{eq:HCH}
H \cdot \Co{\ugamma} \cdot H^{-1} = \Jordan,
\end{equation}
where $J $ is the Jordan normal form consisting of the Jordan blocks $\Jblock{\lambda_1}, \ldots, \Jblock{\lambda_\ellJ} $.

There is some ambiguity for $H$ and we want to make an explicit choice in the following. There are some natural ways to construct a Jordan basis either 
\begin{enumerate}
 \item as the columns of $H^{-1} $ regarding the multiplication by $\Co{\ugamma} $ from the left or
 \item as the rows of $H $ regarding the multiplication by $\Co{\ugamma} $ from the right.
\end{enumerate}

Recall that $\chi_C(X)=X^n+C_1X^{n-1} + \ldots + C_{n-1}X+C_n $ is the characteristic polynomial of $\Co{\ugamma} $. Now, the vector 
\[
\col{j}:=
\begin{pmatrix}
C_{n-1}+C_{n-2}\lambda_j + C_{n-3}\lambda_j^2+ \ldots + C_1 \lambda_j^{n-2}+\lambda_j^{n-1} \\
 \vdots\\
 C_2+C_1 \lambda_j+\lambda_j^2\\
 C_1+\lambda_j\\
 1
\end{pmatrix} \in \C^n
\]
is an eigenvector as in (1). 

We will write 
\[
c_j^{(k)} \defeq ( \frac{\partial}{\partial \, \lambda_j})^k \, c_j
\]
for the $k^{\text{th}} $ derivative with respect to $\lambda_j $. It is easily observed that the $k_j $ columns give an $n \times k_j $-matrix
\[
\K_j \defeq \left(
\begin{array}{c|c|c|c|c}
&&&&\\
c_j & c_j^{(1)} & \frac{1}{2!} c_j^{(2)} & \ldots & \frac{1}{(k_j-1)!} c_j^{(k_j-1)}\\
&&&&
\end{array}\right )  \in \C^{n \times k_j} 
\]
satisfying
\begin{equation}\label{eq:KJ}
\Co{\ugamma} \cdot \K_j = \K_j \cdot \Jblock{\lambda_j}.
\end{equation}

Hence
\[
K^{-1} \defeq 
\left(
\begin{array}{c|c|c|c}
&&&\\
 \K_1 & \K_2 & \cdots & \K_\ellJ\\
&&&
\end{array}\right) \in \GL_n(\C)
\] 
is a possible choice for \eqref{eq:HCH} writing $K $ instead of $H $.

Considering the eigenvector problem for the multiplication by $\Co{\ugamma} $ from the right, we see that $r_j:= \big( 1, \lambda_j, \lambda_j^2, \ldots, \lambda_j^{n-1} \big) $ is an eigenvector for the eigenvalue $\lambda_j=\exp(2 \pi i \gamma_j) $. 

As before, writing
\[
r_j^{(k)} \defeq (\frac{\partial}{\partial \lambda_j})^k \, r_j \in (\C^n)^\ast
\]
for the row vector, we obtain a $k_j \times n $-matrix
\[
\Hm_j \defeq 
\left( \begin{array}{ccc}
& \frac{1}{(k_j-1)!} r_j^{(k_j-1)} & \\
& \vdots & \\
  & \frac{1}{(2)!} r_j^{(2)} &  \\
   & r_j^{(1)} &  \\
    & r_j & 
\end{array} \right) \in \C^{k_j \times n}
\]
satisfying
\begin{equation}\label{eq:HJ}
 H_j \cdot \Co{\ugamma} = \Jblock{\lambda_j} \cdot H_j .
\end{equation}
Therefore, we get another solution for \eqref{eq:HCH}, namely the generalized Vandermonde type matrix which decomposes vertically into blocks of size $k_1\times n, \ldots, k_\ellJ \times n $ as
\begin{equation}\label{eq:HvdM}
H:=
\left(
\begin{array}{ccc}
&H_1& \\ \hline
&H_2& \\ \hline
&\vdots &\\ \hline
& H_\ellJ
\end{array}\right)
\end{equation}
We will choose this isomorphism in the following. Note that
\begin{enumerate}
 \item the last line of $H $ is the row $r_\ellJ $ for the eigenvalue $\lambda_\ellJ=1 $, hence the last line of $H $ is $(1,1,\ldots, 1)$.
 \item the last line of $K^{-1} $ decomposes into blocks of size $k_1, k_2, \ldots, k_\ellJ $ -- the algebraic multiplicities of the eigenvalues -- and then reads
 \begin{equation}\label{eq:lastrowK}
 \text{last line of $K^{-1} $}= (1,0, \ldots, 0 \mid 1, 0, \ldots, 0 \mid \cdots \mid 1,0, \ldots, 0).
 \end{equation}
\end{enumerate}
Since both $H $ and $K $ satisfy \eqref{eq:HCH}, we conclude that $A\defeq H \cdot K^{-1} $ commutes with the Jordan matrix $J $. Consequently, $A $ also decomposes into blocks
\[
\def\arraystretch{1.2}
A= 
\left(
\begin{array}{ccc}
\boxed{
\begin{array}{c}
 A_1
\end{array}} \\
& \ddots & \\
& &
\boxed{
\begin{array}{c}
 A_\ellJ
\end{array}}
\end{array}\right)
\]
of blocks $A_j $ of size $k_j \times k_j $ which have the form
\begin{equation}\label{eq:Aj}
A_j = 
\left(
\begin{array}{cccccc}
a_{j,0} & a_{j,1} & a_{j,2} & \cdots & a_{j, k_j-1} \\
& a_{j,0} & a_{j,1} &  \cdots & a_{j, k_j-2} \\
&& a_{j,0}  & \cdots & a_{j, k_j-3} \\ 
&&& \ddots & \vdots \\
&&&& a_{j,0}
\end{array}\right) = \sum_{k=0}^\infty a_{j,k} \cdot N^k
\end{equation}
for the nilpotent matrix $N $ being the standard Jordan block  of size $k_j\times k_j $ with $0 $ in the diagonal and $1 $ in the first super-diagonal.

The entries $a_{j,k} $ for $0 \le k \le k_j-1 $ are computed as
\[
 a_{j,k}=\frac{1}{k! (k_j-1)!} r_j^{(k)} \cdot c_j^{(k_j-1)}.
\]
Note that $c_j=C \cdot \sigma_j $ with 
\[
C  \defeq
\begin{pmatrix}
 1& C_1 &C_2& \cdots & C_{n-1} \\
 &1&C_1& \cdots & C_{n-2} \\
 && \ddots& \ddots\\
 &&&&1
\end{pmatrix} \ , \ 
\sigma_j  \defeq 
\begin{pmatrix}
 \lambda^{n-1}\\\vdots \\ \lambda \\ 1
\end{pmatrix}
\]
Now, for any $0 \le a \le b \le k_j-1 $, we have
\[
\frac{1}{a!} \cdot r_j^{(a)}=(0, \ldots, 0, 1, {a+1 \choose 1} \lambda, {a+2 \choose 2} \lambda^2, \ldots, {n-1 \choose n-1-a} \lambda^{n-1-a})
\]
with $a $ zeroes in front, and 
\[
\frac{1}{b!} \cdot \sigma_j^{(b)}= {}^t ({n-1 \choose n-1-b} \lambda^{n-1-b}, {n-2 \choose n-2-b} \lambda^{n-2-b}, \ldots, {b+1 \choose 1} \lambda, 1, 0, \ldots, 0 )
\]
with $b $ zeroes in the end. A tedious but easy calculation shows that
\begin{multline*}
 \frac{1}{a!b!} r_j^{(a)} c_j^{(b)} =
 \frac{1}{a!b!} r_j^{(a)} \cdot C \cdot \sigma_j^{(b)} = \\
M_{0} \lambda_j^{n-a-b-1}+ M_{1} C_1 \lambda_j^{n-a-b-2} + \ldots + M_{n-1-a-b} C_{n-1-a-b}
\end{multline*}
with
\[
M_{\kappa} \defeq 
\sum_{\ellJ=a}^{n-b-1-\kappa} {\ellJ \choose a}\cdot {n-1-\kappa-\ellJ \choose b} = {n-\kappa \choose a+b+1},
\]
the last equality is usually known as the Chu-Vandermonde-formula. Hence, we obtain
\begin{multline*}\label{eq:rcab}
\frac{1}{a! \cdot b!} r_j^{(a)} c_j^{(b)} = \\
{n \choose a+b+1} \lambda_j^{n-(a+b+1)} + {n-1 \choose a+b+1} C_1 \lambda_j^{n-1-(a+b+1)} + \ldots + C_{n-1-(a+b+1)} = \\
 \frac{1}{(a+b+1)!} \cdot \chi_C^{(a+b+1)}(\lambda_j)
\end{multline*}
and deduce the formula for the coefficients of the blocks $A_j $ in \eqref{eq:Aj}:
\begin{equation}\label{eq:ajk}
a_{j,k}= \frac{1}{(k_j+k)!} \cdot \chi_C^{(k_j+k)}(\lambda_j) .
\end{equation}
Note that it is zero when $k_j+k \ge n+1 $. 

Let us use the following notation. For each $\lambda_j $, we write $\chi_C(X)= (X-\lambda_j)^{k_j} \cdot \chi_j(X) $, i.e.
\begin{equation} \label{eq:chijdef}
\chi_j(X) \defeq \prod_{\ellJ \neq j} (X-\lambda_\ellJ)^{k_\ellJ} \ .
\end{equation}
Then
\begin{equation}\label{eq:aijchi}
 a_{j,k} = \frac{1}{k!} \chi_j^{(k)}(\lambda_j) \ .
\end{equation}

For later purposes, we compute the blocks $A_j^{-1} $ of the inverse matrix $A^{-1} $ using the following lemma which is easily proved by multiplying the two matrices.
\begin{lemma}
 Let $M $ be an $k \times k $-matrix of the form $M=\sum_{\nu=0}^\infty \frac{1}{\nu!} \frac{d^\nu}{dX^\nu} a(X) \cdot N^\nu $ for some polynomial $a(X) \in \C[X] $. Then the inverse matrix is given by
 \[
 M^{-1} = \sum_{\nu=0}^\infty \frac{1}{\nu!} \cdot \frac{d^\nu}{dX^\nu} (a(X)^{-1}) \cdot N^\nu \ .
 \]
\end{lemma}

Let us introduce the following
\begin{notation}\label{not:taylor}
Let $f(X) \in \C(X) $ be a rational function. Then the \emph{Taylor vector of $f $ of order $<k $} is given by
\[
\Tayl{f(X)}{k} \defeq
\begin{pmatrix}
\frac{1}{(k-1)!} f^{(k-1)}(X)\\  
\vdots\\
f'(X)\\  
f(X)\\  
\end{pmatrix}
\in \C(X)^k .
\]
We will often write $\TaylX{f(X)}{k}{\lambda} $ for the value of the term at the point $\lambda $.
The \emph{truncated Taylor vector} is 
\begin{equation}\label{eq:Tayltrunc}
\Tayltr{f(X)}{k} \defeq 
\begin{pmatrix}
\frac{1}{(k-1)!} f^{(k-1)}(X)\\  
\vdots\\
f'(X)
\end{pmatrix}
\in \C(X)^{k-1} .
\end{equation}
The \emph{Taylor matrix of $f $ of size $k \times k $} is the matrix 
\[
\Ttayl{f}{k}(X) \defeq \sum_{j=0}^\infty \frac{1}{j!} f^{j}(X) \cdot N^j
\]
with the nilpotent matrix $N $ as above.
\end{notation}

\begin{remark}
 For $f,g \in \C(X) $, we have
 \begin{align*}
 \Ttayl{f}{k}(X) \cdot \Tayl{g}{k}(X) & = \Tayl{fg}{k}(X) \\
  \Ttayl{f}{k}(X) \cdot \Ttayl{g}{k}(X) & = \Ttayl{fg}{k}(X) .
 \end{align*}
 \end{remark}
 
We see that the matrix $A=HK^{-1} $ has blocks of the form
\begin{equation}\label{eq:Ablocks}
 A_j \defeq \TtaylX{\frac{\chi_C(X)}{(X-\lambda_j)^{k_j}}}{k_j}{\lambda_j} 
\end{equation}
and hence its inverse matrix $A^{-1} $ is block-diagonal with blocks
\begin{equation}\label{eq:Ainvblocks}
 A^{-1}_j \defeq \TtaylX{\frac{(X-\lambda_j)^{k_j}}{\chi_C(X)}}{k_j}{\lambda_j} 
\end{equation}

We choose $H $ as the transition matrix producing the Jordan normal form. Note that the restriction of $H $ to $\im(1-\Co{\ugamma}) $ induces an isomorphism
\[
H:\im(1-\Co{\ugamma}) \isoto \C^{n-1}
\]
since the eigenspace of $\Co{\ugamma} $ for the eigenvalue $1 $ is one-dimensional.
We let $\vi_\rglambda $ be as in \eqref{eq:vi1}. Let us also denote by $\mathrm{pr}_{n-1}:\C^n \to \C^{n-1} $ the projection to the first $n-1 $ factors and $\iota_{n-1}:\C^{n-1} \to \C^n $ the inclusion $x \mapsto (x,0) $. We obtain the following.
\begin{proposition}\label{ref:propdiag}
Under Assumption \ref{ass:1EW}, we have the following isomorphism of quivers:
\begin{equation} \label{eq:diaquiv}
\xymatrix@!C{
 {\im(1-\Co{\ugamma})} \ar@<-.5ex>[r]_(.6){\iota_0}  & {\C^n} \ar@<-.5ex>[l]_(.4){1-\Co{\ugamma}} \ar@<.5ex>[r]^(.4){1-\Co{\ueta}\Co{\ugamma}^{-1}}  & 
{\im(\Co{\ugamma}-\Co{\ueta})} \ar@<.5ex>[l]^(.6){\iota_\rglambda} \\
{\C^{n-1}} \ar@<-.5ex>[r]_(.6){V''_0} \ar[u]^{\simeq}_{H^{-1}|_{\C^{n-1}}} & {\C^n} \ar@<-.5ex>[l]_(.4){U_0''} \ar[u]^{\simeq}_{H^{-1}} \ar@<.5ex>[r]^(.4){U_\rglambda''}  & 
{\C} \ar@<.5ex>[l]^(.6){V''_\rglambda} \ar[u]^{\simeq}_{\vi_\rglambda}
} 
\end{equation}
with
\begin{align}
U_0''  & \defeq \mathrm{pr}_{n-1} \circ ( 1-\Jordan), \\
V_0'' & \defeq \iota_{n-1}, \label{eq:diqui} \\
U_\rglambda'' & \defeq
\fre \cdot A^{-1} \Jordan^{-1}, \notag \\
V_\rglambda'' & \defeq  {}^t v \defeq {}^t(v_1 \mid v_2 \mid \ldots \mid v_{\ellJ-1} \mid v_\ellJ), \notag
\end{align}
where
\[
\fre \defeq (1,0, \ldots, 0 \mid 1,0, \ldots, 0 \mid \ldots \mid 1,0, \ldots, 0 ) \in (\C^{n-1})^\vee
\]
is the row vector with blocks of length $k_1, \ldots, k_\ellJ $, and $v $ is the columns vector with analogous block structure and blocks of the form
\begin{equation}\label{eq:deffrv}
v_j \defeq \TaylX{\chi_E(X)}{k_j}{\lambda_j} .
\end{equation}
\end{proposition}
\begin{proof}
 The statements on $U_0''$ and $V_0''$ are obvious by \eqref{eq:HCH}. We have
 \begin{align*}
 U_\rglambda''= & \varphi_\rglambda^{-1}(\Co{\ugamma}-\Co{\ueta}) \Co{\ugamma}^{-1} H^{-1} =
 \varphi_\rglambda^{-1}(\Co{\ugamma}-\Co{\ueta}) H^{-1} J^{-1} =\\
&  \varphi_\rglambda^{-1}
\left(\begin{array}{c|c}
  \mbox{$\text{\Large $0$}$} &
 \begin{array}{c}
 E_n-C_n\\
 \vdots \\
 E_1-C_1 
\end{array}
 \end{array}\right) \cdot K^{-1} A^{-1} J^{-1} = \fre \cdot A^{-1} J^{-1} 
 \end{align*}
due to the shape of the last row of $K^{-1} $, see \eqref{eq:lastrowK}.

Additionally, due to the shape of $H $ as in \eqref{eq:HvdM} we see that the vector $H\cdot {}^t(E_n,\ldots, E_1) $ decomposes into blocks of length $k_1, \ldots, k_{\ellJ-1}, k_\ellJ $, the $j^{\text{th}} $ block reading
\[
H_j \cdot 
\begin{pmatrix}
 E_n\\ \vdots\\ E_2\\ E_1
\end{pmatrix} = 
\begin{pmatrix}
 \frac{1}{(k_j-1)!} \cdot \chi_E^{(k_j-1)}(\lambda_j) - {n \choose {k_j-1}} \lambda_j^{n-k_j-1}\\
 \vdots\\
\chi_E'(\lambda_j) - {n \choose 1} \lambda_j^{n-1}\\
\chi_E(\lambda_j) - \lambda_j^n
\end{pmatrix}
\]
The analogous computation holds for the coefficients $C_j $ instead of $E_j $. Since $\lambda_j $ is a root of multiplicity $k_j $ for the polynomial $\chi_C(X) $, we obtain
\[
H_j \cdot 
\begin{pmatrix}
C_n\\ \vdots\\ C_1
\end{pmatrix}
=
\begin{pmatrix}
- {n \choose {k_j-1}} \lambda_j^{n-k_j-1}\\
 \vdots\\
- {n \choose 1} \lambda_j^{n-1}\\
- \lambda_j^n
\end{pmatrix}.
\]
The claim about $V_\rglambda'' $ follows.
\end{proof}

\section{The Stokes matrices for the unramified confluent hypergeometric equation}

In this section we want to compute the Stokes matrices for the unramified confluent hypergeometric system $\Hypl{\ualpha}{\ubeta}{\glambda} $ with $\ualpha=(\alpha_1, \ldots, \alpha_n)$ and $\ubeta=(\beta_1, \ldots, \beta_{n-1}) $ under the genericity Assumption \ref{ass:generic}.
Due to Katz's result \eqref{eq:FouH1} $j_\ast \Hypabl \simeq   \Fou{\big( \jint \Hypl{1, -\ubeta}{-\ualpha}{\glambda^{-1}} \big)} $, we are led to apply the previous results to the regular singular hypergeometric system $\Hypl{\ugamma}{\ueta}{\rglambda} $ with
\begin{equation}\label{eq:ueta}
\ugamma \defeq (-\ubeta,1) \ , \  \ueta \defeq -\ualpha \text{ and } \rglambda \defeq \glambda^{-1} .
\end{equation}
Note that Assumption \ref{ass:1EW} is satisfied.

Let $F $ be the perverse sheaf of solutions of $\jint\Hypl{\ugamma}{\ueta}{\rglambda} $. After fixing $\fra\defeq i \rglambda$ and $\frb \defeq \rglambda^{-1} $, we associate to it the quiver
\[
\xymatrix{
{\Phi_0(F)} \ar@<-.5ex>[r]_{v_0} & {\Psi(F)} \ar@<-.5ex>[l]_{u_0} \ar@<.5ex>[r]^{u_\rglambda}  & 
{\Phi_\rglambda(F)} \ar@<.5ex>[l]^{v_\rglambda}}.
\]

Let $M=\Fou(\jint\Hypl{\ugamma}{\ueta}{\rglambda} ) $ be the Fourier transform. We have $\Sigma=\{0,\rglambda \} $ with $0<_\frb \rglambda $ and hence the local isomorphism class of $M $ at $\infty $ is represented by the Stokes matrices (see section \ref{sec:Stokesgeneral}):
\begin{equation}\label{eq:S+S-}
S_+= 
\left( 
\begin{array}{c|c}
 1 & u_0v_\rglambda \\ \hline
0 & 1
\end{array}\right) \ , \ 
S_-=
\left( 
\begin{array}{c|c}
 1-u_0v_0 &  0\\ \hline
-u_\rglambda v_0 & 1-u_\rglambda v_\rglambda
\end{array}\right) \ , \ 
\end{equation}
both of them being understood as linear maps
\[
\Phi_0(F) \oplus \Phi_\rglambda(F) \to \Phi_0(F) \oplus \Phi_\rglambda(F).
\]
Choosing basis for the vector spaces involved gives actual matrices with complex coefficients. The choice of the basis does not change the equivalence class of the pair and we can slightly abuse notation and write
\[
[S_+,S_-] \in \Stmatr{\dim \Phi_0(F), \dim \Phi_\rglambda(F)}
\]
for this class. We apply this formula to the quiver \eqref{eq:quivHyp1} of Proposition \ref{prop:quiverR}. To obtain explicit matrices, we consider the isomorphic quivers given in Proposition \ref{ref:propcompanion} and Proposition \ref{ref:propdiag} respectively.

\subsection{The companion representation}

We get the following result

\begin{theorem}\label{thm:companion}
Let $\ualpha \defeq (\alpha_1, \ldots, \alpha_n) $ and $\ubeta \defeq (\beta_1, \ldots, \beta_{n-1}) $ be generic and consider the polynomials
\begin{align*}
\chi_B(X) \defeq \prod_{j=1}^{n-1} (X-\exp(-2\pi i \beta_j)) &= X^{n-1}+B_1 X^{n-2}+ B_2 X^{n-3} + \ldots + B_{n-1}, \\
\chi_A(X) \defeq \prod_{j=1}^n (X- \exp(-2\pi i \alpha_j)) &= X^n + A_1 X^{n-1} + A_2 X^{n-2}+ \ldots + A_n.
\end{align*}
With the choice of $\fra=i \rglambda$ and $\frb=\rglambda^{-1} $ as base direction and orientation, the equivalence class of Stokes matrices for the hypergeometric system $\Hypl{\ualpha}{\ubeta}{\glambda} $ at infinity is represented by the pair
\[
 S_+  = 
 \left(
 \begin{array}{c|c}
1_{n-1} & x \\[.2cm] \hline
  0 & 1
\end{array}\right)
  \text{ and }
   S_- = 
  \left(
 \begin{array}{c|c}
  \mbox{$\Co{-\ubeta}$} & 0 \\[.2cm] \hline
  y & \exp(2\pi i \lambda)
\end{array}\right)
\]
with
\begin{align*}
x &=   \left( \begin{array}{l}
 A_n\\
 A_{n-1}-B_{n-1}\\
 \vdots \\
 A_2-B_2
\end{array}\right) - (A_1-B_1) \cdot 
 \left( \begin{array}{l}
B_{n-1}\\
B_{n-2}\\
\vdots \\
B_1
\end{array}\right) 
 \in \C^{n-1},\\
 y &= 
\big( (-1)^n\exp(2\pi i \sum_{j=1}^{n-1} \beta_j) ,0 \ldots, 0 \big) \in (\C^{n-1})^\vee, \\
\lambda &= 1- \sum_{j=1}^n \alpha_j + \sum_{j=1}^{n-1} \beta_j.
\end{align*}
 \end{theorem}
\begin{proof}
We use the companion representation of the quiver given in Proposition \ref{ref:propcompanion} for $\ugamma, \ueta $ as in \eqref{eq:ueta}. In comparison to the notation above, we have $\chi_E(X)=\chi_A(X) $ and $\chi_C(X)=(X-1) \cdot \chi_B(X) $. We deduce that
\[
C_n+\ldots+C_k=-B_{k-1}
\]
for $k=2, \ldots, n $ and $C_1=B_1-1 $, as well as $A_k=E_k $. The statement then is an easy computation.
\end{proof}

\subsection{The Jordan representation} \label{sec:JordanStokes}

In the presentation of the Stokes matrices above, the formal monodromy is included a priori in these matrices. In the literature, the formal monodromy most often is given as an extra, isolated term. In the next result, we want to adapt our presentation to this common practice. Additionally, we want to give representatives of the Stokes matrices in a \emph{normal form} in order to better understand the ambiguities arising from different choices of basis. To this end, let us compute the Stokes matrices according to \eqref{eq:S+S-} using the quiver presentation of Proposition \ref{ref:propdiag}. We easily compute $1-U_0''V_0'' $, $U_0''V_\rglambda $ and $U_\rglambda'' V_0'' $. Additionally,
\begin{multline*}
1-U_\rglambda''V_\rglambda''=1- \big[ \varphi_\rglambda^{-1}(\Co{\ugamma}-\Co{\ueta}) \Co{\ugamma}^{-1}H^{-1} \cdot H 
\begin{pmatrix}
 E_n-C_n\\ \vdots\\ E_1-C_1
\end{pmatrix} \big]=\\
1-\big[ \varphi_\rglambda^{-1} 
\left(\begin{array}{c|c}
  \mbox{$\text{\Large $0$}$} &
 \begin{array}{c}
 E_n-C_n\\
\\
 \vdots \\
 E_1-C_1 
\end{array}
 \end{array}\right)
 \left(
 \begin{array}{ccccc}
 -C_n^{-1}C_{n-1} & 1 & & \\
 -C_n^{-1}C_{n-2} & 0 &1 & \\
 \vdots & && \ddots \\
 -C_n^{-1}C_1 & && & 1 \\
 -C_n^{-1} &&&&0 
\end{array}\right)
\begin{pmatrix}
 E_n-C_n\\ \vdots\\ E_1-C_1
\end{pmatrix} \big]=\\
1-\big[ \varphi_\rglambda^{-1} \big( C_n^{-1} (C_n-E_n) \big)
\begin{pmatrix}
 E_n-C_n\\ \vdots\\ E_1-C_1
\end{pmatrix} \big]=C_n^{-1}E_n=\exp(2\pi i \sum_{j=1}^n (\eta_j-\gamma_j)).
\end{multline*}
Due to \eqref{eq:ueta} we have
\[
\sum_{j=1}^n (\eta_j-\gamma_j) = -1-\sum_{j=1}^n \alpha_j +\sum_{j=1}^{n-1} \beta_j \equiv \lambda \text{ mod } \Z,
\]
hence we obtain the equivalence class (recall that $\fre \defeq (1,0, \ldots, 0 \mid 1,0, \ldots, 0 \mid \ldots \mid 1,0, \ldots, 0 ) \in (\C^{n-1})^\vee $):
\begin{equation*}
 [S_+,S_-]=
\left[
\left( 
\begin{array}{c|c}
1_{n-1} & \pr_{n-1} \circ (1-\Jordan) \cdot v \\ \hline
0 & 1
\end{array}\right) \ , \ 
\left( 
\begin{array}{c|c}
   \mbox{$\pr_{n-1} \circ \Jordan \circ \iota_{n-1}$} & 0 \\ \hline
-\fre \cdot A^{-1} \Jordan^{-1} \iota_{n-1} & \exp(2\pi i \lambda)
\end{array}\right) 
\right]
\end{equation*}
We will write
\[
\widetilde{\Jordan} \defeq \pr_{n-1} \circ \Jordan \circ \iota_{n-1} \text{ and } 
\widetilde{A} \defeq \pr_{n-1} \circ A \circ \iota_{n-1}
\]
for the matrices arising by cutting off the last row and column -- hence reducing the block corresponding to the eigenavlue $1 $ by one row and column.
Executing the base change (see Definition \ref{def:equivrel}) given by
\[
\left( 
\begin{array}{c|c}
   -e^{-2\pi i \lambda} \cdot \widetilde{A}^{-1} \widetilde{\Jordan}^{-1} & 0 \\ \hline
0 & 1
\end{array}\right) \cdot S_{\pm} \cdot 
\left( 
\begin{array}{c|c}
- e^{2\pi i \lambda} \cdot \widetilde{\Jordan} \widetilde{A} &0\\ \hline
0 & 1
\end{array}\right)
\]
we obtain the equivalent description (note that $A $ and $\Jordan $ commute):
\begin{equation} \label{eq:SJordan1}
[S_+,S_-]=
\left[
\left( 
\begin{array}{c|c}
1_{n-1} & 
e^{-2\pi i \lambda} \cdot \widetilde{A}^{-1} \widetilde{\Jordan}^{-1}
     (\widetilde{\Jordan}-1) \cdot v \\ \hline
0 &  1
\end{array}\right) \ , \ 
M_\infty^- \cdot 
\left( 
\begin{array}{c|c}
1_{n-1}  &  0 \\ \hline
\fre & 1
\end{array}\right) 
\right]
\end{equation}
where we used the following
\begin{notation}\label{notation:formon}
 We define
 \[
 M_\infty^- \defeq 
 \left( 
\begin{array}{c|c}
\widetilde\Jordan  &0  \\ \hline
0 & e^{2\pi i \lambda}
\end{array}\right) .
\]
Its conjugacy class is the \emph{formal monodromy} at infinity in clockwise orientation with respect to the origin (see Figure \ref{fig:sectors}).
\end{notation}

\subsubsection{Ambiguity respecting the Jordan decomposition} \label{sec:ambigJordan}

The representatives $(S_+,S_-) $ in \eqref{eq:SJordan1} are determined up to the base change of Definition \ref{def:equivrel}. We want to describe a normal form. A natural requirement is to fix the Jordan form of the representative of the formal monodromy. Let us introduce the following notion.
\begin{definition}
 A base change $(S_+,S_-) \mapsto (AS_+A^{-1},AS_-A^{-1}) $ for some $A \in \Delta_{0, \rglambda} $ as in Definition \ref{def:equivrel} is called \emph{adapted to the Jordan form} if its block $A_{0,0} \in \GL_{n-1}(\C) $ respects the decomposition
 \[
 \C^{n-1} = \bigoplus_{j=1}^n \ker(\Jordan-\lambda_j)^{\infty}
 \]
 into the generalized eigenspaces of $\Jordan $ (in other words if $A_{0,0} $ is a block diagonal matrix with respect to this decomposition). 
\end{definition}

\begin{proposition}\label{prop:xy}
 We fix a numbering of the eigenvalues. Then the pair $[S_+,S_-] $ of the Stokes matrices for $\opHyp{\ualpha}{\ubeta} $ has a representative of the form
 \[
 [S_+,S_-] = 
 \left[
\left( 
\begin{array}{c|c}
1_{n-1}  & 
z \\ \hline
0 &  1
\end{array}\right) \ , \ 
M_\infty^-\cdot 
\left( 
\begin{array}{c|c}
1_{n-1}  & 0  \\ \hline
\fre  & 1
\end{array}\right) 
\right]
 \]
The vector $z \in \C^{n-1} $ then is uniquely determined by the equivalence class $[S_+,S_-] $. We will call such a representative the \emph{normal form of} $[S_+,S_-] $.
\end{proposition}
\begin{proof}
Existence has already been observed above. For uniqueness, let $A $ be any base change. Then we have
\begin{equation}
 [S_+,S_-] = 
 \left[
\left( 
\begin{array}{c|c}
1_{n-1}  & 
A_{\rglambda,\rglambda}^{-1} A_{0,0} z \\ \hline
0 &  1
\end{array}\right) \ , \ 
A \cdot M_\infty^- \cdot A^{-1} \cdot 
\left( 
\begin{array}{c|c}
1_{n-1}  & 0  \\ \hline
A_{\rglambda, \rglambda} \fre A_{0,0}^{-1} & 1
\end{array}\right) \label{eq:SJordan11}
\right]
 \end{equation}
We see that multiplying $A $ with a constant in $\C^\times $ does not change anything, hence we can assume $A_{\rglambda,\rglambda}=1 $. The pair \eqref{eq:SJordan11} is again of normal form if
\[
AM_\infty^- A^{-1} = M_\infty^- \text{ and } \fre  A_{0,0}^{-1} = \fre \ .
\]
From the first condition we deduce that $A $ is adapted to the Jordan form in the following way. If we decompose $A_{0,0} $ into blocks according to the Jordan decomposition of $\C^n $, each block $A_{0,0}^{(\lambda_j, \lambda_k)} $ not in the block diagonal satisfies $\Jblock{\lambda_j} A_{0,0}^{(\lambda_j, \lambda_k)} =A_{0,0}^{(\lambda_j, \lambda_k)} \Jblock{\lambda_k} $ for distinct eigenvalues $\lambda_j \neq \lambda_k $, hence equals the zero matrix. The second condition is then equivalent to ask that each block of $A_{0,0}^{-1} $ (associated to the generalized eigenspace for the eigenvalues $\lambda_j $) has the vector $(1,0, \ldots, 0) $ as its first row. 

The requirement $AM_\infty^- A^{-1}=M_\infty^-$ however, tells us that $A_{0,0} $ commutes with $\widetilde{\Jordan} $ and hence each block of $A_{0,0}^{-1} $ is of the form $\sum_{r=0}^\infty a_r N^r $ for some $a_r \in \C $. Both conditions together imply that $A_{0,0} $ is the identity matrix. 

\end{proof}

\subsubsection{The Jordan representation}

The following result describes the normal form of the Stokes matrices in the unramified case. 
\begin{theorem}\label{thm:jordan}
  Let $\ualpha \defeq (\alpha_1, \ldots, \alpha_n) $ and $\ubeta \defeq (\beta_1, \ldots, \beta_{n-1}) $ be generic and consider the polynomials
\[
\chi_B(X)   \defeq \prod_{j=1}^{n-1} (X-\exp(-2\pi i \beta_j)) \text{\quad and \quad} \chi_A(X)  \defeq \prod_{j=1}^n (X- \exp(-2\pi i \alpha_j))
\]
of degree $n-1 $ and $n $ respectively. We assume that the roots of $\chi_B $ are subdivided into
%\[
%\lambda_1 \defeq e^{-2\pi i \beta_1} = \ldots = e^{-2\pi i \beta_{\kappa_1}}, 
%\lambda_2 \defeq e^{-2\pi i\beta_{k_1+1}} = \ldots = e^{-2\pi i\beta_{\kappa_1+\kappa_2}}, \ldots , 
%\lambda_\ell  \defeq e^{-2\pi i\beta_{n-k_\ell}} = \ldots = e^{-2\pi i\beta_{n-1}}
%\]
\begin{align*}
\lambda_1 & \defeq e^{-2\pi i \beta_1} = \ldots = e^{-2\pi i \beta_{\kappa_1}} \\ 
\lambda_2 & \defeq e^{-2\pi i\beta_{k_1+1}} = \ldots = e^{-2\pi i\beta_{\kappa_1+\kappa_2}} \\ 
\vdots & \hspace*{2cm}\vdots \\ 
\lambda_\ell & \defeq e^{-2\pi i\beta_{n-k_\ell}} = \ldots = e^{-2\pi i\beta_{n-1}}
\end{align*}
with pairwise different $\lambda_j $, and we denote by $\kappa_j$ the corresponding multiplicity so that $\chi_B(X)=\prod_{j=1}^\ell (X- \lambda_j)^{\kappa_j} $. With the choice of $\fra=i\rglambda $ and $\frb=\rglambda^{-1} $ as base direction and orientation, the equivalence class of Stokes matrices for the hypergeometric system $\Hypl{\ualpha}{\ubeta}{\glambda} $ at infinity is given by its normal form ($M_\infty^- $ being the formal monodromy -- Notation \ref{notation:formon}):
\[
[S_+,S_-] = \left[
\left( 
\begin{array}{c|c}
1_{n-1}  & 
z \\ \hline
0 &  1
\end{array}\right) \ , \ 
M_\infty^- \cdot 
\left( 
\begin{array}{c|c}
1_{n-1}  &  0 \\ \hline
\fre & 1
\end{array}\right) 
\right]
 \]
with $\fre=(1,0, \ldots, 0 \mid \ldots \mid 1,0, \ldots 0) $ with blocks of size $\kappa_1, \ldots, \kappa_{\ell} $ and $z = e^{-2\pi i \lambda} \cdot {}^t(z_1, \ldots, z_\ell) $ with
\begin{equation}\label{eq:finalresult}
z_j = \TaylX{\frac{\chi_A(X)}{X}  \cdot \frac{(X-\lambda_j)^{\kappa_j}}{\chi_B(X)}}{\kappa_j}{\lambda_j} \text{\quad for $j=1, \ldots, \ell$} \ ,
\end{equation}
and $\lambda = 1- \sum_{j=1}^n \alpha_j + \sum_{j=1}^{n-1} \beta_j$.
\end{theorem}
\begin{proof}
We consider the Jordan representation of the quiver as in \eqref{eq:SJordan1}
\[
[S_+,S_-]=
\left[
\left( 
\begin{array}{c|c}
1_{n-1} & 
e^{-2\pi i \lambda} \cdot\widetilde{A}^{-1} \widetilde{\Jordan}^{-1}
 (\widetilde{\Jordan}-1) \cdot v \\ \hline
0 &  1
\end{array}\right) \ , \ 
M_\infty^- \cdot 
\left( 
\begin{array}{c|c}
1_{n-1}  &0  \\ \hline
\fre & 1
\end{array}\right) 
\right]
\]
with $\ugamma, \ueta $ as in \eqref{eq:ueta}. Let us assume that we ordered the $\lambda_j $ so that at most $\lambda_\ell=1 $. Then we are in the situation of \eqref{eq:lambdamult}. More precisely, let us distinguish the cases whether some of the $\beta_j $ are integers or not. Note that with respect to the notation \eqref{eq:defchiC}, we have $\chi_E(X) =\chi_A(X) $ and $\chi_C(X) = (X-1) \cdot \chi_B(X)$.

\medskip\noindent
{$1^{\text{st}} $ case: $1 \in \{ \lambda_j \mid j=1, \ldots, \ell \} $.}\\ \noindent
Then we have (in the notation of \eqref{eq:lambdamult}) $\ellJ=\ell $ and $k_\ellJ=\kappa_\ell+1 $ (cp. \eqref{eq:ueta}). Accordingly, the matrices $\widetilde{A} $ and $\widetilde{J} $ have blocks $\widetilde{A}_j  = A_j $ and $\widetilde{J}_j  = J_j $ for $j=1, \ldots, \ell-1 $ and 
$\widetilde{A}_\ell $, $\widetilde{J}_\ell $ arise from $A_\ell $ and $J_\ell $ (the latter being the Jordan block for the eigenvalue $1 $ of size $k_q =\kappa_\ell+1 $) by deleting the last row and column. 

We have to determine the upper right block of $S_+ $, namely $\widetilde{A}^{-1} (1-\widetilde{\Jordan}^{-1}) \cdot v $. Since all factors decompose into blocks, it is enough to consider each block individually. For $j=1, \ldots, \ell-1 $, we have
\begin{align*}
\widetilde{A}^{-1}_j &= \TtaylX{\frac{(X-\lambda_j)^{\kappa_j}}{\chi_C(X)}}{\kappa_j}{\lambda_j} & \text{see \eqref{eq:Ainvblocks}}\\
(1-\widetilde{J}^{-1})_j &= \TtaylX{1-X^{-1}}{\kappa_j}{\lambda_j} \\
v_j &= \TaylX{\chi_E(X)}{\kappa_j}{\lambda_j} & \text{see \eqref{eq:deffrv}}.
\end{align*}
Hence we obtain
\begin{equation}\label{eq:zcomp}
 \widetilde{A}_j^{-1} (1-\widetilde{\Jordan}^{-1})_j \cdot v_j  = 
 \TaylX{\frac{(X-\lambda_j)^{\kappa_j}}{(X-1)\cdot \chi_B(X)} \cdot \frac{X-1}{X} \cdot \chi_A(X)}{\kappa_j}{\lambda_j}
 \end{equation}
as asserted.
For $j=\ell $ we obtain the same result but we have to consider the truncation -- see \eqref{eq:Tayltrunc} -- and recalling that $k_\ell=\kappa_\ell+1 $:
\begin{multline*}
A_\ell^{-1} (1-J^{-1})_j v_j = \TayltrX{\frac{(X-1)^{k_\ell}}{\chi_B(X)} \cdot \frac{1}{X} \cdot \chi_A(X)}{k_\ell}{1}\\
=\TayltrX{(X-1) \cdot \frac{(X-1)^{\kappa_\ell}}{\chi_B(X)} \cdot \frac{\chi_A(X)}{X} }{\kappa_\ell+1}{1}
=\TaylX{\frac{(X-1)^{\kappa_\ell}}{\chi_B(X)} \cdot \frac{\chi_A(X)}{X} }{\kappa_\ell}{1} \ ,
\end{multline*}
where in the last equation we used the fact that
\[
\TayltrX{(X-1)\cdot f(X)}{k+1}{1} = \TaylX{f(X)}{k}{1} \ .
\]

\medskip\noindent
{$2^{\text{nd}} $ case: $1 \not\in \{ \lambda_j \mid j=1, \ldots, \ell \} $.}\\ \noindent
Then we have $\ellJ=\ell+1 $ and $k_\ellJ=1 $. The eigenvalue $1 $ (artificially added for the regular singular system) of $J $ has a $1\times 1 $-Jordan block which is distinguished by the truncation. Therefore, only the blocks for $j=1, \ldots, \ellJ-1=\ell $ remain to be considered. The computation is the same as above \eqref{eq:zcomp}.
\end{proof}

Let us emphasize that the representation in Theorem \ref{thm:jordan} contains only rational expressions in the eigenvalues $e^{-2\pi i \alpha_j} $ and $e^{2\pi i \beta_j} $ of the local monodromies or their inverses. In comparison to Duval-Mitschi's result, no Gamma-function appears. As a consequence, we obtain the
\begin{corollary}
In the non-resonant unramified case (Assumption \ref{ass:generic}),  there is a presentation of the Stokes matrices for $\Hypabl $ such that $S_+, S_- $ are defined over $\Q(e^{2\pi i \alpha_1}, \ldots, e^{2\pi i \alpha_n}, e^{2\pi i \beta_1}, \ldots, e^{2\pi i \beta_{n-1}}) $.

In particular, if $\alpha_j \in \Q $ for $j=1, \ldots, n $ and $\beta_j\in \Q $ for $j=1, \ldots, n-1 $ and Assumption \ref{ass:generic} is satisfied, then there is a presentation of the Stokes matrices such that $S_+, S_- $ are defined over a cyclotomic field of finite degree over $\Q $.
\end{corollary}

\section{Some special cases}

\subsection{The cyclotomic and the real case}\label{sec:cyclo}

As a corollary of the companion representation, we can easily isolate cases where we find representatives with integer coefficients. We assume that $\ualpha \in \Q^n $ and $\ubeta \in \Q^{n-1} $. Note that the hypergeometric module only depends on the classes of the parameters $\alpha_j $ and $\beta_i $ \emph{modulo integers}, i.e. we can assume that $\alpha_j, \beta_i \in [0,1) \cap \Q $.

\begin{definition}
We say that a family of parameters $\ugamma=(\gamma_1, \ldots, \gamma_k) \in ([0,1) \cap\Q)^k $ satisfies the \emph{cyclotomic property} if the following holds: 
\begin{multline*}
\text{If } \frac{v}{w} \in \ugamma \text{ for some relatively prime } v,w \in \Z \Longrightarrow \\
\frac{u}{w} \in \ugamma \text{ for all $u $ such that } \mathrm{gcd}(u,w)=1,
\end{multline*}
and then each value $\tfrac{u}{w} $ with $\mathrm{gcd}(u,w)=1 $ appears with the same mutliplicity among the parameters $\ugamma $.
\end{definition}

The following is another way to think about this definition where we denote by 
\[
\Phi_m(X) \defeq \prod_{a \in (\Z/m\Z)^\times} (X- e^{{2\pi i a}/{m}}) \in \Z[X]
\]
the irreducible cyclotomic polynomial for the $m^\text{th} $ cyclotomic field over $\Q $.

\begin{lemma}
 The parameters $\ugamma \in ([0,1) \cap \Q)^k $ satisfy the cyclotomic property if and only if the polynomial
$ \chi(X) \defeq \prod_{j=1}^k (X-e^{-2\pi i \gamma_j}) \in \C[X] $
decomposes into a product of powers of irreducible cyclotomic polynomials
\[
\chi(X)= \prod_{k=1}^r \big( \Phi_{w_k}(X) \big)^{\nu_k} 
\]
for finitely many $w_k, \nu_k \in \N $.
\end{lemma}

\begin{corollary}[to Theorem \ref{thm:companion}]\label{cor:cyclo}
If $\ualpha, \ubeta $ are generic and both satisfy the cyclotomic property, there is a representative $[S_+,S_-] $ of the Stokes matrices for $\Hypl{\ualpha}{\ubeta}{\glambda} $ with integer entries: $S_+, S_- \in \mathrm{GL}_n(\Z) $. Additionally, there is a representative of the formal mondromy with entries in $\Z $.
\end{corollary}
\begin{proof}
The entries of the representative of Theorem \ref{thm:companion} only contain the coefficients of the polynomials $\chi_A(X) $ and $\chi_B(X) $. Under the genericness assumption and the cyclotomic property, we have $\chi_A(X) = \prod_{k=1}^s \Phi_{t_k}(X) \in \Z[X] $ and $\chi_B(X) = \prod_{k=1}^r \big( \Phi_{w_k}(X) \big)^{\nu_k} \in \Z[X] $. In addition, we have $\lambda \in \Z $ and $\sum_{j=1}^{n-1} \beta_j \in \Z $. The determinant of $\Co{-\ubeta} $ is $B_{n-1}=\pm 1 $. The diagonal blocks of $S_- $ give the formal monodromy.
\end{proof}

\medskip
In the same spirit, we introduce the following notion:
\begin{definition}
We say that a family of parameters $\ugamma=(\gamma_1, \ldots, \gamma_k) \in (\C^\times)^k $ satisfies the \emph{complex conjugate property} if the following holds: 
\[
\text{for any $j=1, \ldots, k $ there exists an $i \in \{1, \ldots, k \} $ such that } \gamma_j+\gamma_i \in \Z \ .
\]
(Note, that $j=i $ is allowed, i.e. $\gamma_j \in \tfrac{1}{2} \Z $).
\end{definition}

\begin{corollary}[to Theorem \ref{thm:companion}]\label{cor:real}
For $\ualpha $ and $\ubeta $ generic, both satisfying the complex conjugate property, there is a representative $[S_+,S_-] $ of the Stokes matrices for 
$\Hypl{\ualpha}{\ubeta}{\glambda} $ with real coefficients: $S_+, S_- \in \mathrm{GL}_n(\R) $. The same holds for the formal monodromy.
\end{corollary}
\begin{proof}
 Under the assumptions, the polynomials $\chi_A(X) $ and $\chi_B(X) $ are invariant under complex conjugation, hence have real coefficients. Additionally $\lambda \in \tfrac{1}{2}\Z $ and $\sum_{j=1}^{n-1} \beta_j \in \tfrac{1}{2}\Z $. 
\end{proof}

\begin{remark}
 It would be interesting to understand the relation of this corollary with the existence of a real variation of Hodge structures under similar conditions in the work of Fedorov (see Theorem 2 of \cite{fedorov}, proving a conjecture of A. Corti and V. Golyshev). We plan to address this question in a future work. 
\end{remark}

\subsection{The diagonalizable case and comparison with Duval-Mitschi's Result}

In \cite{DM} the authors compute the Stokes matrices for confluent hypergeometric systems (even in the ramified and reducible case). However, they make the following additional assumption -- the assumption is stated in the \emph{Remarque} on page 29 in \cite{DM}.

\begin{assumption}
 If the generic (in terms of Assumption \ref{ass:generic}) parameters $(\ualpha,\ubeta) $ additionally satisfy the condition:
 \[
 \beta_i \not \equiv \beta_j \text{ mod } \Z \text{ for all } i \neq j ,
 \]
 we call them \emph{generic and diagonalizabe}.
 \end{assumption}

Under this assumption, the eigenvalues $\{ \exp(2\pi i \beta_j)\mid j=1, \ldots, n-1 \} $  are pairwise disjoint and hence the Jordan matrix $\Jordan $ of \eqref{eq:HCH} is a diagonal matrix and the corresponding blocks are of size one. Therefore, the main result simplifies to the following statement.

\begin{theorem}\label{thm:diagonal}
Let $(\ualpha;\ubeta) $ be generic and diagonalizable. Define $\chi_A(X)  \defeq \prod_{j=1}^n (X- e^{-2\pi i \alpha_j}) $ and $\chi_B(X)  \defeq \prod_{j=1}^{n-1} (X-e^{-2\pi i \beta_j})$. 
%\[
%\chi_A(X)  \defeq \prod_{j=1}^n (X- e^{-2\pi i \alpha_j}) \text{\quad and \quad}
%\chi_B(X)  \defeq \prod_{j=1}^{n-1} (X-e^{-2\pi i \beta_j}).
%\]
Then the Stokes matrices of the hypergeometric module $\Hyp{\ualpha}{\ubeta} $ have a representation in the form
\begin{equation}\label{eq:ourS}
[ S_+, S_-]  = [
 \left(
 \begin{array}{c|c}
1_{n-1} & z \\[.2cm] \hline
  0 & 1
\end{array}\right) \ , \ 
  \left(
 \begin{array}{c|c}
  \mbox{$\diag(e^{-2 \pi i \ubeta})$} & 0 \\[.2cm] \hline
0 & e^{2\pi i \lambda}
\end{array}\right) 
\cdot
  \left(
 \begin{array}{c|c}
1_{n-1} & 0 \\[.2cm] \hline
  \fre & 1
\end{array}\right) 
]
\end{equation}
with $\fre=(1,1,\ldots, 1) $ and $z=e^{-2\pi i \lambda} \cdot  {}^t\!(z_1, \ldots, z_{n-1}) \in \C^{n-1} $ with
\[
z_j = \frac{\chi_A(e^{-2\pi i \beta_j})}{e^{-2\pi i \beta_j} \cdot \chi'_B(e^{-2\pi i \beta_j})}
\]
\end{theorem}

\begin{remark}
Though very elementary, let us write down the explicit formulae
\begin{align}
\label{eq:chiA}
\chi_A(e^{-2\pi i \beta_j}) & = \prod_{\ell=1}^n (e^{-2\pi i \beta_j} - e^{-2\pi i \alpha_\ell}) \\
\label{eq:chiB}
\chi_B'(e^{-2\pi i \beta_j}) & = \prod_{\ell \neq j} (e^{-2\pi i \beta_j}-e^{-2\pi i \beta_\ell}).
\end{align}
\end{remark}

\begin{remark}\label{rem:distinct}
In order to compare our result with the existing computations of Duval-Mitschi, let us take a closer look at the situation and choices considered by the latter.
\begin{enumerate}
\item They study (more generally for $q>p $) the hypergeometric operator 
\[
D_{q,p}(\umu; \unu) \defeq (-1)^{q-p} z \prod_{j=1}^p ( z\partial_z + \mu_j) - \prod_{j=1}^q (z\partial_z+\nu_j-1) \ ,
\]
which in the unramified case $p=q-1 $ gives $D_{q,q-1}(\umu;\unu) = -\opHypl{1-\unu}{-\umu}{-1} $ for $\umu \in \C^{q-1} $ and $\unu \in \C^q $. Hence, they consider the case $\glambda=\rglambda=-1 $ -- note, that in \cite[p.50]{DM}, there is a sign typo in the formula for $D_{3,2}$ in comparison to their original definition of $D_{q,p} $ and the general result.

 \item They seperate the formal monodromy from the Stokes matrices. The formal monodromy is induced by counter-clockwise rotation with respect to the origin. 
\end{enumerate}
\end{remark}

\subsubsection{Comparison of the geometric setup}\label{subsec:equrel}

Duval-Mitschi computed the Stokes matrices $\DMS_0 $ and $\DMS_1 $ by constructing formal solutions and explicitly lifting these in sectors. The sectors (unbounded with respect to the radius) they consider are (in the obvious notation they use)
\begin{align*}
 \theta_0 & \defeq \theta(-\frac{3\pi}{2},\frac{\pi}{2}), & \theta_1 & \defeq \theta(-\frac{\pi}{2}, \frac{3\pi}{2}) , & \theta_2 & \defeq \theta(-\frac{\pi}{2},\frac{5\pi}{2}),
\end{align*}
to be read in the universal cover of $\C\smallsetminus \{ 0 \} $. They write $\Sigma_j $ for the constructed basis of solutions in $\theta_j $. We understand these as isomorphisms
\[
\Sigma_j: \C^n \isoto \shl_b
\]
to the stalk of the local system $\shl $ at the base-point $b $ near $0 $. They define the Stokes matrices by
\[
 \DMS_0  \defeq \Sigma_1^{-1} \circ \Sigma_0 \text{ and }
  \DMS_1  \defeq \Sigma_2^{-1} \circ \Sigma_1. 
\]

Since $\rglambda=\glambda=-1 $ and $\fra=-i $, in the notation of \cite{DHMS}, we consider the sectors
\[
H_{\pm \fra} \defeq \{ z \in \C^\times \mid \pm \Re (\fra z) \ge 0 \} = \{ z \in \C^\times \mid \pm \Im(z) \ge 0 \}
\]
where $z $ is a coordinate in the target $\C $ centered at $0 $. The Stokes matrices $S_{\pm} $ result form the trivializations of the enhanced solutions sheaf in these sectors. More precisely, they are the glueing matrices of these trivializations from $H_\fra $ to the one on $H_{-\fra} $ -- see \cite[Section 5.2]{DHMS}. We picture the situation in Figure \ref{fig:sectors}.

\begin{figure}\centering
\begin{tikzpicture}[scale=.7,
	background rectangle/.style={draw=black,dashed,fill=white},
	show background rectangle]
\coordinate (C1) at (-.5,0);
\coordinate (C2) at (.3,.1); %vorher (0,
\coordinate (C3) at (1.5,0);  %vorher (.5,
\coordinate (C4) at (2.6,.1);  %vorher (.5,
\coordinate (C5) at (3.5,-.5);  %vorher (.5,

\begin{scope}[decoration={
    markings,
    mark=at position .7 with {\arrow{>}}}
    ]
\clip (-1.5,-1.5) rectangle (4.2, 2); 
\filldraw[fill=black!30, draw=black, thick] (C3)++(-3.5,0) -- ++(7,0) -- ++(0,3) -- ++(-7,0) -- cycle ;
\draw[thick, postaction=decorate] (C3)++(-3.5,0) -- (C3);
\draw (C2) node[below]{$h_{\frb}$};
\draw[thick, postaction=decorate] (C3)++(3.5,0) -- (C3);
\draw (C4) node[below]{$h_{-\frb}$};
\end{scope}

\filldraw[black] (C3) circle (2pt) node[below]{$\infty$} ;
\draw (1.5,1.4)  node{$H_{-\fra} $};
\draw (1.5,-1.4)  node{$H_{\fra}$};

\begin{scope}[decoration={
    markings,
    mark=at position .99 with {\arrow{>}}}
    ]
\draw[thick,blue,postaction={decorate}]	(0,-1) to[in=-135, out=-235](0,1) ;
\draw[blue] (0.2,.5) node{$S_+ $};
\draw[thick,red,postaction={decorate}]	(-0.3,-1) to[in=235, out=125](-0.3,1) ;
\draw[red] (-0.8,1) node{$\DMS_1 $};

\draw[thick,blue,postaction={decorate}]	(2.7,-1) to[in=-45, out=45](2.7,1) ;
\draw[blue] (2.5,.5) node{$S_- $};
\draw[thick,red,postaction={decorate}]	(3,1) to[in=45, out=-45](3,-1) ;
\draw[red] (3.5,-1.2) node{$\DMS_0 $};

\end{scope}

\end{tikzpicture}
\caption{The sectors and directions of the Stokes matrices in Duval-Mitschi's ($\DMS_{0/1} $) and our notation ($S_{\pm} $). Note that the formal monodromy $M_\infty^- $ of Notation \ref{notation:formon} is associated to a loop around $\infty $ in counter-clockwise orientation in the figure above.}\label{fig:sectors}
\end{figure}
%\begin{equation}\label{eq:ourS}
%[ S_+, S_-]  = [
% \left(
% \begin{array}{c|c}
%1_{n-1} & z \\[.2cm] \hline
%  0 & 1
%\end{array}\right) \ , \ 
%  \left(
% \begin{array}{c|c}
%  \mbox{$\diag(e^{-2 \pi i \ubeta})$} & 0 \\[.2cm] \hline
%0 & e^{2\pi i \lambda}
%\end{array}\right) 
%\cdot
%  \left(
% \begin{array}{c|c}
%1_{n-1} & 0 \\[.2cm] \hline
%  \fre & 1
%\end{array}\right) 
%]
%\end{equation}
%with $\fre=(1,1,\ldots, 1) $ and $z=e^{-2\pi i \lambda} \cdot  {}^t\!(z_1, \ldots, z_{n-1}) \in \C^{n-1} $ with
%\[
%z_j = \frac{\chi_A(e^{-2\pi i \beta_j})}{e^{-2\pi i \beta_j} \cdot \chi'_B(e^{-2\pi i \beta_j})}
%\]

\subsubsection{Comparison of the Stokes matrices}

Our result gives the equivalence class $[S_+,S_-] $ with representatives in normal form as in \eqref{eq:ourS}. We now compare this to the one of Duval-Mitschi with the help of Proposition \ref{prop:xy}. 

Let us state Duval-Mitschi's result on $D_{n,n-1}(\umu;\unu) $ after adapting the parameters to our notation, i.e. replacing
\[
\unu = 1-\ualpha \text{ and } \umu = -\ubeta .
\]
Note that Duval-Mitschi use the term $\lamDM \defeq 1+\sum_{j=1}^{n-1} \mu_j - \sum_{j=1}^n \nu_j $ in their article, which compares as
\begin{equation}\label{eq:lamDM}
\lamDM=1+\sum_{j=1}^{n-1}(-\beta_j) - \sum_{j=1}^n (1-\alpha_j) = -\lambda + 2-n
\end{equation}
to our convention. 

\begin{theorem}[Duval-Mitschi, Th\'eor\`eme 5.1 (a) in \cite{DM}]
 With the notations and choices introduced above, the Stokes matrices $\DMS_0 $ and $\DMS_1 $ for the hypergeometric system $\opHypl{\ualpha}{\ubeta}{-1} $ for generic $(\ualpha,\ubeta) \in (\C^\times)^n \times (\C^\times)^{n-1} $ are
\begin{equation}\label{eq:SDM}
 \DMS_0  =  
 \left(
 \begin{array}{c|c}
1_{n-1} & 0 \\[.2cm] \hline
  v & 1
\end{array}\right) \text{ and }
 \DMS_1 =
  \left(
 \begin{array}{c|c}
1_{n-1}  & w \\[.2cm] \hline
  0 & 1
\end{array}\right)
\end{equation}
where the vectors $v, w $ have the entries
\[
v_j   \defeq 2\pi i \cdot \frac{\prod_{\ell=1, \ell\neq j}^{n-1} \Gamma(1-(\beta_j-\beta_\ell))}{\prod_{\ell=1}^n \Gamma(\alpha_\ell-\beta_j)} \ , \ 
 w_j  \defeq 2\pi i e^{\pi i(\lamDM-\beta_j)} \cdot \frac{\prod_{\ell=1,\ell\neq j}^{n-1} \Gamma(\beta_j-\beta_\ell)}{\prod_{\ell=1}^n \Gamma(1-(\beta_j-\alpha_\ell))} .
\]
\end{theorem}

In terms of Definition \ref{def:equivrel} and Propositon \ref{prop:StokesRH}, Duval-Mitschi's pair corresponds to the equivalence class (cp. Figure \ref{fig:sectors}):
\begin{multline*}
[\DMS_+,\DMS_-] \defeq [\DMS_1, M_\infty^- \cdot (\DMS_0)^{-1}] = \\
\left[
  \left(
 \begin{array}{c|c}
1_{n-1}  & w \\[.2cm] \hline
  0 & 1
\end{array}\right)
\ , \
  \left(
 \begin{array}{c|c}
  \mbox{$\diag(e^{-2 \pi i \ubeta})$} & 0 \\[.2cm] \hline
0 & e^{2\pi i \lambda}
\end{array}\right) 
\cdot
 \left(
 \begin{array}{c|c}
1_{n-1} & 0 \\[.2cm] \hline
  -v & 1
\end{array}\right)\right]
\end{multline*}

Executing the base change (see Definition \ref{def:equivrel}) given by
\[
\left( 
\begin{array}{c|c}
   \diag(-v) & 0 \\ \hline
0 & 1
\end{array}\right) \cdot \DMS_{\pm} \cdot 
\left( 
\begin{array}{c|c}
\diag(-v)^{-1} &0\\ \hline
0 & 1
\end{array}\right)
\]
we obtain the equivalent description in normal form:
\begin{equation}\label{eq:DMSs}
[\DMS_+,\DMS_-] = 
\left[
  \left(
 \begin{array}{c|c}
1_{n-1}  & -v \cdot w \\[.2cm] \hline
  0 & 1
\end{array}\right)
\ , \ 
M_\infty^-
\cdot
 \left(
 \begin{array}{c|c}
1_{n-1} & 0 \\[.2cm] \hline
  \fre & 1
\end{array}\right)\right]
\end{equation}

Due to Proposition \ref{prop:xy}, we know that the pairs \eqref{eq:ourS} and \eqref{eq:DMSs} are equivalent if and only if
\[
-v_j w_j =  e^{-2\pi i \lambda}z_j 
\]
for all $j=1, \ldots, n-1 $. Since the factors are of the same shape for each $j $, it suffices to  consider $j=1 $. With \eqref{eq:chiA} and \eqref{eq:chiB} we get
\begin{multline*}
-e^{-2\pi i \lambda} z_1 = -e^{2\pi i (\beta_1-\lambda)} \cdot
\frac{\prod_{\ell=1}^n (e^{-2\pi i \beta_1}-e^{-2\pi i\alpha_\ell})}{\prod_{\ell=2}^{n-1}(e^{-2\pi i \beta_1}-e^{-2\pi i \beta_\ell})} =\\
-e^{2\pi i(\beta_1-\lambda)} \cdot
\frac{e^{-2\pi i \sum_{\ell=1}^n \alpha_\ell} \cdot e^{-2\pi i n \beta_1} \cdot
\prod_{\ell=1}^n (e^{2\pi i \alpha_\ell}-e^{2\pi i\beta_1})
}{
e^{-2\pi i \sum_{\ell=2}^{n-1} \beta_\ell} \cdot e^{-2\pi i (n-2) \beta_1} \cdot
\prod_{\ell=2}^{n-1}(e^{2\pi i \beta_\ell}-e^{2\pi i \beta_1})} =\\
-e^{-4\pi i \beta_1} \cdot
\frac{\prod_{\ell=1}^n (e^{2\pi i \alpha_\ell}-e^{2\pi i\beta_1})}{\prod_{\ell=2}^{n-1}(e^{2\pi i \beta_\ell}-e^{2\pi i \beta_1})}.
\end{multline*}

On the other hand side, using Euler's formula
\[
\Gamma(x)\Gamma(1-x) = \frac{\pi}{\sin(\pi x)} = \frac{2\pi i e^{\pi i x}}{e^{2\pi i x}-1}
\] 
and \eqref{eq:lamDM}, we obtain
\begin{multline*}
v_1 w_1=(2\pi i)^2 \cdot e^{\pi i (\lamDM-\beta_1)} \cdot
\frac{\prod_{\ell=2}^{n-1} \Gamma(\beta_1-\beta_\ell) \Gamma(1-(\beta_1-\beta_\ell))}{
\prod_{\ell=1}^n \Gamma(1-(\alpha_\ell-\beta_1))\Gamma(\alpha_\ell-\beta_1)} = \\
(2\pi i)^2 \cdot e^{\pi i(-\lambda+2-n-\beta_1)} \cdot
\prod_{\ell=2}^{n-1} \frac{2\pi i e^{\pi i (\beta_1-\beta_\ell)}}{e^{2\pi i (\beta_1-\beta_\ell)}-1}
\cdot
\prod_{\ell=1}^{n} \frac{e^{2\pi i (\alpha_\ell-\beta_1)}-1}{2\pi i e^{\pi i (\alpha_\ell-\beta_1)}}=\\
e^{\pi i(-\lambda+2-n-\beta_1)} \cdot
\prod_{\ell=2}^{n-1} \frac{e^{\pi i (\beta_1+\beta_\ell)}}{e^{2\pi i \beta_1}-e^{2\pi i\beta_\ell}}
\cdot
\prod_{\ell=1}^{n} \frac{e^{2\pi i \alpha_\ell}-e^{2\pi i \beta_1}}{e^{\pi i (\alpha_\ell+\beta_1)}}=
e^{\pi i \Lambda} \cdot \frac{\prod_{\ell=1}^n (e^{2\pi i \alpha_\ell}-e^{2\pi i\beta_1})}{\prod_{\ell=2}^{n-1}(e^{2\pi i \beta_1}-e^{2\pi i \beta_\ell})}
\end{multline*}
with
\[
\Lambda=\overbrace{-1+\sum_{\ell=1}^n\alpha_\ell - \sum_{\ell=1}^{n-1}\beta_\ell}^{-\lambda}
+2-n-\beta_1+(n-2)\beta_1+\sum_{\ell=2}^{n-1}\beta_\ell -\sum_{\ell=1}^n\alpha_\ell -n \beta_1 = -4\beta_1-(n-1)\\
\]
leading to the same result as above. In conclusion, we have proven:
\begin{proposition}
The description of Duval-Mitschi and the one of Theorem \ref{thm:diagonal} are equivalent:
\[
[\DMS_+,\DMS_-]= [S_+, S_-].
\]
\end{proposition}

\section{Extension to the ramified case} \label{sec:ramified}

\subsection{The setting}

In this section we extend our results to the general case with ramification -- still non-resonant. We will mainly focus on the necessary changes in comparison to the unramified case. The results will have a more complicated shape. We compute the Stokes matrices (of the de-ramified hypergeometric equation) in a companion representation and deduce the rationality statements similar to section \ref{sec:cyclo}.

We consider the hypergeometric module $\Hypl{\ualpha}{\ubeta}{\glambda} $ with $\ualpha=(\alpha_1, \ldots, \alpha_n) $, $\ubeta=(\beta_1, \ldots, \beta_m) $ with $n>m $ and we let $d:=n-m $ be the ramification degree to be applied. We impose the following assumption:

\begin{assumption}\label{ass:genericramified}
We assume $\uab $ are non-resonant and $d\alpha_i \not\in \Z $ for all $i=1, \ldots, n $. Furthermore we assume that $\Hypl{\ualpha}{\ubeta}{\glambda} $ is not Kummer induced -- see \cite[Kummer Recognition Lemma 3.5.6]{katz}.
\end{assumption}

From Katz's Theorem we know that under these assumptions we have
\[
 \jint[d]^\ast \Hypabl \simeq \Fou{\big(\jint [d]^\ast \Hypl{\frac{1}{d}, \frac{2}{d}, \ldots, \frac{d}{d}, - \ubeta}{-\ualpha}{(d^d)/\glambda} \big)}
 \]
 and furthermore this module is irreducible. As usual (cp. \cite{DM}), we will compute the Stokes matrices of the \emph{de-ramified} module $\jint[d]^\ast \Hypabl $ on the affine space $\Af^1_\tau $ with the coordinate $\tau $ such that $\tau^d=z $. 
 
\subsection{Quiver of the perverse sheaf}

The first task is to understand the quiver associated to the solution perverse sheaf $G $ associated to the regular singular module
$\jint [d]^\ast \Hypl{\frac{1}{d}, \frac{2}{d}, \ldots, \frac{d}{d}, - \ubeta}{-\ualpha}{(d^d)/\glambda} $.
Let us write $\mathbb{0} \in \Af^1_\tau $ from now on in order to distinguish this point from the index $k=0 $. We know from Katz's Theorem that $G $ is an irreducible perverse sheaf with singular support on $\widetilde{\Sigma} :=[d]^{-1}(\{0,\frac{d^d}{\glambda}\}) = \{ \bO \} \cup \frac{d}{\sqrt[d]{\glambda}} \cdot \mu_d $ where $\sqrt[d]{\glambda} $ is any chosen $d $-th root and $\mu_d $ is the group of $d$-th roots of unity.  Let us write $\rglambda:=d/\sqrt[d]{\glambda} $ and let $\zeta $ be the primitive element $\zeta:=\exp(2\pi i/d) \in \mu_d$. 

We choose the direction/orientation as $\fra:=\rglambda \cdot e^{\frac{\pi i}{2}-i \varepsilon} $ and $\frb:=\rglambda^{-1} e^{-i \varepsilon} $ for a small $\varepsilon>0 $ -- see Figure \ref{fig:zetas}. These induce the total ordering $<_\frb $ on $\widetilde{\Sigma} $. For any $d \ge 2 $, we have
\begin{equation}\label{eq:orderram}
\rglambda >_\frb \rglambda \zeta^{-1} >_\frb \rglambda \zeta >_\frb \rglambda \zeta^{-2} >_\frb \rglambda \zeta^2 >_\frb \ldots >_\frb \rglambda \zeta^{\lfloor \frac{d}{2} \rfloor} \ .
\end{equation}
The position of the value $\mathbb{0} $ in this ordering depends on $d \text{ mod } 4 $:
\begin{align}
 d=4e  & \Longrightarrow \rglambda \zeta^{-e} >_\frb \mathbb{0} >_\frb \rglambda \zeta^e \\
 d=4e+1,2 \text{ or } 3 & \Longrightarrow  \rglambda \zeta^e >_\frb  \mathbb{0} >_\frb \rglambda \zeta^{-(e+1)} .
\end{align}
Let $\nu \in \{\lfloor \frac{d}{2} \rfloor, \ldots, d-1 \} $ be the exponent such that $\sigma \zeta^\nu <_\frb \mathbb{0} <_\frb \sigma\zeta^{\nu+2} $, e.g. if $d=4e+1,2,3 \not\equiv 0 \text{ mod } 4 $, then $\nu=d-e-1 $.

\begin{figure}\centering
\begin{tikzpicture}[scale=.7,
	background rectangle/.style={draw=black,dashed,fill=white}, 
	show background rectangle]
\begin{scope}[decoration={
    markings,
    mark=at position .7 with {\arrow{>}}}
    ]
\clip (-1.2,-2.2) rectangle (3.1, 1.8); 
\node[inner sep=0pt] (zeta9) at (.98,-.7)
    {\includegraphics[scale=1.2]{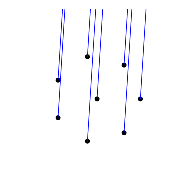}};
 \end{scope}
\end{tikzpicture}
\qquad
\begin{tikzpicture}[scale=.7,
	background rectangle/.style={draw=black,dashed,fill=white}, 
	show background rectangle]
\begin{scope}[decoration={
    markings,
    mark=at position .7 with {\arrow{>}}}
    ]
\clip (-1.2,-2.2) rectangle (3.1, 1.8); 
\node[inner sep=0pt] (zeta9) at (.98,-.7)
    {\includegraphics[scale=1.2]{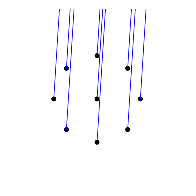}};
 \end{scope}
\end{tikzpicture}
\qquad
\begin{tikzpicture}[scale=.7,
	background rectangle/.style={draw=black,dashed,fill=white}, 
	show background rectangle]
\begin{scope}[decoration={
    markings,
    mark=at position .7 with {\arrow{>}}}
    ]
\clip (-1.2,-2.2) rectangle (3.1, 1.8); 
\node[inner sep=0pt] (zeta9) at (.98,-.7)
    {\includegraphics[scale=1.2]{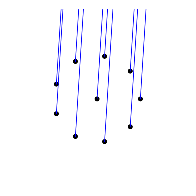}};
 \end{scope}
\end{tikzpicture}
\caption{The lines in direction $\fra $ inducing the ordering of the singular values $\widetilde{\Sigma}=\{\bO\} \cup \rglambda \mu_d $. In the pictures, we chose $\rglambda=1 $ for convenience. From left to right: $d=7,8,9 $.}\label{fig:zetas}
\end{figure}

The quiver associated to $G $ has the form:
 \[
 \xymatrix{
{\Phi_\bO(G)} \ar@<-.5ex>[r]_(.5){v_\bO} & {\Psi(G)} \big( \ar@<-.5ex>[l]_(.5){u_\bO} \ar@<.5ex>[r]^(.35){u_k}  & 
{\Phi_{\sigma \zeta^k}(G)} \big)_{k=0, \ldots, d-1} \ar@<.5ex>[l]^(.65){v_k} 
}
 \]
Since $G $ is irreducible, it follows from Corollary \ref{cor:middleperv} that the latter is of the special form
  \begin{equation}\label{eq:quivtilde}
 \xymatrix{
{\im(1-\widetilde{T}_\bO)} \ar@<-.5ex>[r]_(.6){\iota_\bO} & {\Psi(G)} \big( \ar@<-.5ex>[l]_(.4){1-\widetilde{T}_\bO} \ar@<.5ex>[r]^(.3){1-\widetilde{T}_k}  & 
{\im(1-\widetilde{T}_k)} \big)_{k=0, \ldots, d-1}  \ar@<.5ex>[l]^(.7){\iota_k}}
 \end{equation}
with the monodromy operators $\widetilde{T}_\bO $ around $\bO $ and $\widetilde{T}_k $ around $\rglambda \zeta^k $ associated to the following choice of generators of the fundamental group fixed by $\fra $ and $\frb $ (cp. \cite[Lemma 4.11]{DHMS}): first, we choose a base-point $\widetilde{b} $ in $\Af^1 \smallsetminus \ell_{\widetilde{\Sigma}} $, say $\widetilde{b} $ in the sector given by $\rglambda \zeta^\nu $ and $\rglambda \zeta^{\nu+1} $ with absolute value $|\widetilde{b}|>|\rglambda | $. For each $z=\rglambda \zeta^k \in \widetilde{\Sigma} $, we denote by $\widetilde{g}_k $ a path inside $\Af^1 \smallsetminus \ell_{\widetilde{\Sigma}} $ from $\widetilde{b} $ to a nearby point of $z $, followed by a counter-clockwise roation around $z $ and returning back to $\widetilde{b} $ inside $\Af^1 \smallsetminus \ell_{\widetilde{\Sigma}} $ -- see Figure \ref{fig:fundgroupram}. The same construction gives a path $\widetilde{g}_\bO $ around $\bO $. Let $b=\widetilde{b}^d $.

%In order to determine the monodromy operators $\widetilde{T}_0 $ and $\widetilde{T}'_k $, let us first consider the following free generators of the fundamental group 
%\[
%\pi_1(\Af^1\smallsetminus \widetilde{\Sigma},\widetilde{b}) = \langle \widetilde{g}_0, \widetilde{g}'_k \mid k=0,\ldots, d-1 \rangle
%\]
%as sketched in Figure \ref{fig:fundgroupram} for some base-point $\widetilde{b} $ as in the Figure.

\begin{figure}\centering
\begin{tikzpicture}[scale=1,
	background rectangle/.style={draw=black,dashed,fill=white}, 
	show background rectangle]
\coordinate (C1) at (2,1);
\coordinate (C2) at (.5,0); %vorher (0,
\coordinate (C3) at (-.5,0);  %vorher (.5,
\coordinate (C4) at (-1.3,0);  %vorher (.5,
\coordinate (C5) at (3,0);  %vorher (.5,

\begin{scope}[decoration={
    markings,
    mark=at position .7 with {\arrow{>}}}
    ]
\clip (-1.2,-2.2) rectangle (3.1, 1.8); 

\node[inner sep=0pt] (fundgr) at (.98,-.1)
    {\includegraphics[scale=1.2]{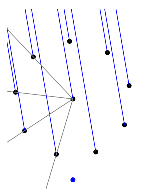}};
    
\draw[red,postaction={decorate}] (.2,.75) circle (4pt);
\draw[red,postaction={decorate}] (1,-.1) circle (4pt);
%\draw[blue,postaction={decorate}] (-.25,1.18) circle (5pt);

\node at (.9,-.4) {$\bO $};
%\node at (.7,0) {\footnotesize $\widetilde{g}_0 $};
%\node at (1.2,1) {\footnotesize $\widetilde{g}'_k $};
\node at (-.2,.8) {\footnotesize $\rglambda \zeta^k $};
\node at (.8,-1.35) {\footnotesize $\rglambda \zeta^\nu $};
\node at (1.8,-1.35) {\footnotesize $\rglambda \zeta^{\nu+1} $};

\node at (1.3,-2) {$\widetilde{b}$};
%\node at (1.3,-2) {$\widetilde{h}_k$};
\node at (2.5,.1) {$\rglambda$};

%\coordinate (C2) at (1.2,-.35); 
\coordinate (C2) at (1,-1.7); 
\coordinate (C4) at (.2,.6);  
\coordinate (C1) at (.8,-2);  

\draw [red,postaction={decorate}] (C2) 
  .. controls ++(90:1) and ++(-90:.5) .. (1,-.25); %(C1)
\draw [red,postaction={decorate}] (C2) 
  .. controls ++(190:1) and ++(-90:.5) .. (C4); %(C1)
% .. controls ++(190:1) and ++(90:-2) .. (C4);
%  .. controls ++(90:1) and ++(115:1) .. (C2);

 \end{scope}
\end{tikzpicture}
\qquad
\begin{tikzpicture}[scale=1,
	background rectangle/.style={draw=black,dashed,fill=white}, 
	show background rectangle]
\coordinate (C1) at (2,1);
\coordinate (C2) at (.5,0); %vorher (0,
\coordinate (C3) at (-.5,0);  %vorher (.5,
\coordinate (C4) at (-1.3,0);  %vorher (.5,
\coordinate (C5) at (3,0);  %vorher (.5,

\begin{scope}[decoration={
    markings,
    mark=at position .7 with {\arrow{>}}}
    ]
\clip (-1.2,-2.2) rectangle (3.1, 1.8); 

%\node[inner sep=0pt] (fundgr) at (.98,-.1)
%    {\includegraphics[scale=.4]{fundgrram4.png}};
\node[inner sep=0pt] (fundgr) at (.98,-.1)
    {\includegraphics[scale=1.2]{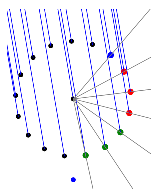}};
    
%\draw[red,postaction={decorate}] (1,-.13) circle (5pt);
%\draw[blue,postaction={decorate}] (-.25,1.18) circle (5pt);

\node at (1.1,-.3) {$\bO $};
\node at (1.7,.65) {\footnotesize $z $};
\node at (2.2,-.05) {$\rglambda$}; %neu
\node at (1,-1.7) {$\widetilde{b}$};
 \end{scope}
\end{tikzpicture}
\caption{Left: The choice of paths $\widetilde{g}_\bO $ and $\widetilde{g}_k $ inducing the monodromies of the quiver. The grey lines help to understand $[d]_\ast \widetilde{g}_k $. Right: An example for the values of $o(k)=3 $ -- the green points -- and $i(k)=3 $ -- the red points -- for $z=\rglambda \zeta^k $.}\label{fig:fundgroupram}
\end{figure}

Let $L $ be the local system associated to $\Hypl{\frac{1}{d}, \frac{2}{d}, \ldots, \frac{d}{d}, - \ubeta}{-\ualpha}{(d^d)/\glambda} $ on $\Af^1 \smallsetminus \{ 0, d^d/\glambda \} $ and $\rho_L:\pi_1(\Af^1 \smallsetminus \{0, d^d/\glambda \}, b)\to \mathrm{GL}_n(\C) $ the corresponding representation of the fundamental group. The local system $\widetilde{L}=[d]^{\ast}L $ with $G|_{\Af^1 \smallsetminus \widetilde{\Sigma}}=\widetilde{L}[1] $ is associated to the representation of the fundamental group
\[
\rho_{\widetilde{L}}: \pi_1(\Af^1 \smallsetminus \widetilde{\Sigma}, \widetilde{b}) \stackrel{[d]_\ast}{\to} 
\pi_1(\Af^1 \smallsetminus \{0, \frac{d^d}{\glambda} \}, b) \stackrel{\rho_L}{\to} \mathrm{GL}_n(\C) \ .
\]
We use the notation $\pi_1(\Af^1 \smallsetminus \{ 0, d^d/\glambda \}, b)=\langle g_0, g_\rglambda \rangle $ as in section \ref{subsec:mdrmy}, such that the monodromies of $L $ are given by $T_0 $ and $T_\rglambda $ as in Corollary \ref{cor:mdrmy}. The monodromy operators to be used in the quiver \eqref{eq:quivtilde} are $\widetilde{T}_\bO = \rho_{\widetilde{L}}(\widetilde{g}_\bO) $ and $\widetilde{T}_k = \rho_{\widetilde{L}}(\widetilde{g}_k) $. Obviously, we have $[d]_\ast \widetilde{g}_\bO = g_0^d $ and due to Corollary \ref{cor:mdrmy}, we deduce that
\begin{equation}
 \widetilde{T}_\bO = T_0^d = \Co{\ugamma}^d \ .
\end{equation}
In order to understand $\widetilde{T}_k $, we introduce the following functions counting the number of sectors the paths $\widetilde{g}_k $ cross outside $o(k) $ or inside $i(k) $ the circle of radius $\rglambda $. For $k \in \{0, \ldots, d-1 \} $, we define
\[
o(k) := \left\{
\begin{array}{ll}
 \# \{ y \in \rglambda \mu_d \mid \bO<_\frb y <_\frb \rglambda\zeta^k \text{ and } \Im(\rglambda^{-1}y)<0 \} & \text{ if $\bO<_\frb \rglambda\zeta^k$},\\
 -\# \{ y \in \rglambda \mu_d \mid \rglambda\zeta^k<_\frb y <_\frb \bO  \text{ and } \Im(\rglambda^{-1}y)<0 \} & \text{ if $\rglambda\zeta^k<_\frb \bO$},\\
\end{array} \right.
\]
and
\[
i(k) := \left\{
\begin{array}{ll}
0 & \text{ for $\Im(\zeta^k)\le 0 $},\\
k+d-(\nu+1)-o(k) \ge 0 & \text{ for $\Im(\zeta^k)>0 $ and $\rglambda \zeta^k>_\frb \bO$},\\
-(\nu-k)+o(k) \le 0  & \text{ for $\Im(\zeta^k)>0 $ and $\rglambda \zeta^k<_\frb \bO$}.\\
\end{array} \right.
\]
cp. Figure \ref{fig:fundgroupram}. With these notations, we see that for each $k=0, \ldots, d-1 $, we get
\[
[d]_\ast(\widetilde{g}_k) = \big((g_1g_0)^{o(k)} \cdot g_0^{i(k)} \big) \cdot g_1 \cdot \big((g_1g_0)^{o(k)} \cdot g_0^{i(k)} \big)^{-1} \ .
\]
Let us define $S_k \in \mathrm{GL}_n(\C) $ by (recall that $\rho_L $ is an anti-homomorphism):
\begin{equation}\label{eq:Sz}
S_k := \big( \rho_L( (g_1g_0)^{o(k)} \cdot g_0^{i(k)} ) \big)^{-1}= \big( \Co{\ugamma}^{i(k)} \cdot ( \Co{\ugamma} \Co{\ueta}\Co{\ugamma}^{-1})^{o(k)} \big)^{-1}\ .
\end{equation}
Then
\[
\widetilde{T}_k = \rho_L([d]_\ast \widetilde{g}_k) = S_k \cdot \Co{\ueta}\Co{\ugamma}^{-1} \cdot S_k^{-1} \ .
\]

\begin{lemma}
Let us write $\ugamma := ( \frac{1}{d}, \frac{2}{d}, \ldots, \frac{d}{d}, - \ubeta ) $ and $\ueta:= -\ualpha $. Under Assumption \ref{ass:genericramified}, the quiver of the irreducible perverse sheaf associated to the regular singular module $\jint [d]^\ast \Hypl{\frac{1}{d}, \frac{2}{d}, \ldots, \frac{d}{d}, - \ubeta}{-\ualpha}{(d^d)/\glambda} $ is (up to isomorphism) given by
  \begin{equation}\label{eq:quiverram1}
  \xymatrix@C2.5cm{
 {\im(1-\Co{\ugamma}^d)} \ar@<-.5ex>[r]_(.6){\iota_\bO}  & {\C^n} \big( \ar@<-.5ex>[l]_(.4){1-\Co{\ugamma}^d} \ar@<.5ex>[r]^(.33){1-S_k\Co{\ueta}\Co{\ugamma}^{-1}S_k^{-1}}  & 
{\im \big( S_k(\Co{\ugamma}-\Co{\ueta}) \big)} \big)_{k=0, \ldots, d-1} \ar@<.5ex>[l]^(.7){\iota_k}
} 
\end{equation}
where $\iota_\bO $ and $\iota_k $ are the inclusions and for each $k=0, \ldots, d-1 $, the isomorphism $S_k $ is defined by \eqref{eq:Sz}.
\end{lemma}

Let us use the same notation $\chi_C $ and $\chi_E $ as in \eqref{eq:defchiC} and \eqref{eq:defchiE}. In order to give a better understanding of the quiver \eqref{eq:quiverram1}, let us remark that the eigenvalues of $\Co{\ugamma} $ are exactly the values $\exp(2\pi i \gamma_j) $, and that for each individual eigenvalue, the eigenspace is one-dimensional -- there is exactly one Jordan block per eigenvalue. Consequently, all $d $-th roots of unity $\mu_d $ are eigenvalues of $\Co{\ugamma} $ each with one-dimensional eigenspace in this case. Then $\Co{\ugamma}^d $ has $1 $ as eigenvalue with dimension of the eigenspace being $\dim\ker (1-\Co{\ugamma}^d)=d $. It follows that
\begin{equation}\label{eq:dimimcod}
\dim \im (1-\Co{\ugamma}^d) = n-d=m \ .
\end{equation}
The matrix $1-\Co{\ugamma}^d $ has the following shape:
\begin{equation}\label{eq:Cogammad}
 1-\Co{\ugamma}^d = \big( 
 e_1-e_{d+1}, e_2-e_{d+2}, \ldots, e_m-e_n, D \big) \in \mathrm{Hom}(\C^n,\C^n),
%   \left(
% \begin{array}{cccc|ccc}
% 1&&&& \\
% &1&&&&\hspace*{.2cm} \mbox{\Large $D $}& \hspace*{.3cm} \\
% &&\ddots&&\\ \hline
% -1&&&&\\
% &-1&&&\\
% && \ddots&&&\hspace*{.2cm} \mbox{\Large $D' $}& \hspace*{.3cm} \\
% &&&-1&
%\end{array}\right) 
\end{equation}
with $e_j $ the standard basis vectors in $\C^n $ and $D $ an $n \times m $-matrix. Note that the entries of $D $ are polynomial expressions in the coefficients of the characteristic polynomial $\chi_C $ of $\Co{\ugamma} $. We see that $\text{span}(e_1-e_{d+1}, e_2-e_{d+2}, \ldots, e_m-e_n) \subset \im (1-\Co{\ugamma}^d)$ and due to \eqref{eq:dimimcod}, equality holds. Hence we obtain an isomorphism
\begin{equation}\label{eq:vi0ram}
\vi_\bO:\C^m \stackrel{\cong}{\to} \im(1-\Co{\ugamma}^d) \ , \ e_j \mapsto e_j-e_{d+j} \ .
\end{equation}

As for the right hand side of the quiver \eqref{eq:quiverram1}, let us fix a $k\in \{ 0, \ldots, d-1 \} $. Recall that we have an isomorphism \eqref{eq:vi1}
\begin{equation}\label{eq:visigram}
\vi:\C \isoto \im(\Co{\ugamma}-\Co{\ueta}) \ , \ 1 \mapsto {}^t(E_n-C_n, E_{n-1}-C_{n-1}, \ldots, E_1-C_1) \ .
\end{equation}
We can understand the morphism $1-S_k\Co{\ueta}\Co{\ugamma}^{-1}S_k^{-1}=S_k(\Co{\ugamma}-\Co{\ueta})\Co{\ugamma}^{-1}S_k^{-1} $ as follows:
\begin{equation}\label{eq:Cougammak}
 \xymatrix@C2cm{
 \C^n \ar[r]^{\Co{\ugamma}^{-1}S_k^{-1}} & \C^n \ar[r]^(.3){\Co{\ugamma}-\Co{\ueta}} & \im(\Co{\ugamma}-\Co{\ueta}) \ar[r]^(.45){S_k}_(.45){\cong} & \im( S_k(\Co{\ugamma}-\Co{\ueta}))\\
 \C^n \ar@{=}[u] \ar[rr]_{U'_k} && \C \ar[u]_{\cong}^{\vi} \ar[ur]^(.3){\cong}_{S_k \vi} 
 }
\end{equation}
by defining
\begin{equation}\label{eq:Usk}
 U'_k := (0 \, 0 \, \ldots \, 0 \, 1) \cdot \Co{\ugamma}^{-1}S_k^{-1} \in \mathrm{Hom}(\C^n,\C) \ .
\end{equation}
We deduce that $\vi_\bO $ and $S_k\vi $ for $k=0, \ldots, d-1 $ induce an isomorphism between the quiver \eqref{eq:quiverram1} and
%\begin{equation}\label{eq:quiverram2}
%   \xymatrix@C2.5cm{
%\C^m \ar@<-.5ex>[r]_(.5){V_0}  & {\C^n} \big( \ar@<-.5ex>[l]_(.5){U_0} \ar@<.5ex>[r]^(.5){U'_k}  & 
%\C \big)_{k=0, \ldots, d-1} \ar@<.5ex>[l]^(.5){V'_k}
%} 
%\end{equation}
%with
\begin{align}\label{eq:quiverram2}
  & \xymatrix@C2.5cm{
\C^m \ar@<-.5ex>[r]_(.5){V_0}  & {\C^n} \big( \ar@<-.5ex>[l]_(.5){U_0} \ar@<.5ex>[r]^(.5){U'_k}  & 
\C \big)_{k=0, \ldots, d-1} \ar@<.5ex>[l]^(.5){V'_k}
} \text{ with}
\\ \notag
 &U'_\bO  = \vi_\bO^{-1} (1-\Co{\ugamma}^d) \in \mathrm{Hom}(\C^n,\C^m),\\ \notag
 & V'_\bO  = \vi_\bO=( e_1-e_d , e_2-e_{d+2}, \ldots, e_m-e_n) \in \mathrm{Hom}(\C^m, \C^n), \\ \notag
& U'_k = (0 \, 0 \ldots \, 0 \, 1) \cdot \Co{\ugamma}^{-1}S_k^{-1} \in \mathrm{Hom}(\C^n,\C), \\ \notag
& V'_k = \Co{\ugamma}^{-1}S_k^{-1} \cdot {}^t(E_n-C_n, \ldots, E_1-C_1) \in \mathrm{Hom}(\C,\C^n) \ .
\end{align}
Note that the first $m $ columns of $U'_\bO $ are the $m \times m $-identity matrix. In summary, we have obtained the following
\begin{proposition}\label{prop:quiverram}
The quiver associated to $\jint [d]^\ast \Hypl{\frac{1}{d}, \frac{2}{d}, \ldots, \frac{d}{d}, - \ubeta}{-\ualpha}{(d^d)/\glambda} $ is (up to isomorphism) given by the quiver  \eqref{eq:quiverram2}.
\end{proposition}

\subsection{The Stokes matrices}

In order to describe the Stokes matrices, let us write $\Phi_\bO := \C^m $ and $\Phi_k := \C $ for the vector spaces in the quiver \eqref{eq:quiverram2}, $k=0, \ldots, d-1 $. 

\begin{remark}
In \cite{DM} a different presentation of the Stokes information is given. The affine line $\Af^1_\tau $ is covered by smaller subsectors each of them  containing only one Stokes-direction. This is a usual procedure since these sectors allow to compute the asymptotic developments of formal solutions with standard methods. However, two closed sectors of width $\pi $ suffice in order to obtain the full information -- cp. \cite[section 5]{DHMS}. In other words, one can multiply all Stokes matrices for the smaller sectors within one of the larger sectors without losing any information. 
\end{remark}

We can now state the result in a companion representation for the ramified case. The Stokes matrices will have the shape
\[
S_\pm: \C^n=\Phi_\bO \oplus \bigoplus_{k=0}^{d-1} \Phi_k \longrightarrow \C^n=\Phi_\bO \oplus \bigoplus_{k=0}^{d-1} \Phi_k \ .
\]
For any $S:\C^n \to \C^n $, we will write $S^{j,k}:\Phi_j \to \Phi_k$ for $j,k \in \{ \bO, 0, \ldots, d-1 \} $ for the corresponding blocks. We will simply write $j <_\frb k $ for these indices and the induced ordering -- e.g. $\bO <_\frb k $ means $\bO <_\frb \sigma \zeta^k $ for $k=0, \ldots, d-1 $.  The Stokes matrices $S_\pm $ are block upper/lower triangular in the sense that
\begin{equation}\label{eq:blocktriangram}
S_+^{j,k}=0 \text{ for } j <_\frb k \text{\quad and \quad} S_-^{j,k}=0 \text{ for } j >_\frb k \ ,
\end{equation}
i.e. $(S_+,S_-) \in U_{\widetilde{\Sigma}} \times L_{\widetilde{\Sigma}} $ in the definitions of section \ref{sec:ambiguity}. We will write $[S_+,S_-] $ for the equivalence class (cp. Definition \ref{def:equivrel}). Applying \cite{DHMS}, we obtain the extension of our main result in the possibly ramified case.
\begin{theorem}\label{thm:companionram}
Let $\ualpha \defeq (\alpha_1, \ldots, \alpha_n) $ and $\ubeta \defeq (\beta_1, \ldots, \beta_m) $ satisfy Assumption \ref{ass:genericramified}. Consider the polynomials
\begin{align*}
\chi_B(X) \defeq \prod_{j=1}^{m} (X-\exp(-2\pi i \beta_j)) &= X^m+B_1 X^{m-1}+ B_2 X^{m-2} + \ldots + B_m, \\
\chi_C(X) \defeq \prod_{j=1}^n (X- \exp(-2\pi i \alpha_j)) &= X^n + C_1 X^{n-1} + C_2 X^{n-2}+ \ldots + C_n,
\end{align*}
and $\chi_E(X) \defeq (X^d-1) \cdot \chi_B(X) = X^n + E_1 X^{n-1} + E_2 X^{n-2}+ \ldots + E_n $.
With the choice of $\fra, \frb $ as above and the induced ordering $<_\frb $, the equivalence class of Stokes matrices for the hypergeometric system $[d]^\ast \Hypl{\ualpha}{\ubeta}{\glambda} $ at infinity is represented by the pair $[S_+,S_-] $ satisfying \eqref{eq:blocktriangram} and
\[
\begin{array}{lll}
 S_+^{j,j} & =1_{\dim(\Phi_j)} & \text{for all } j \in \{ \bO,0, \ldots, k \} \\
 S_+^{j,k} & = U'_j V'_k & \text{for } j>_\frb k \\
S_-^{j,j} & = 1_{\dim(\Phi_j)} - U'_jV'_j & \text{for all } j \in \{ \bO,0, \ldots, k \} \\
S_-^{j,k} &=-U'_jV'_k & \text{for } j<_\frb k \ ,
\end{array}
\]
where $\dim\Phi_\bO=d $, $\dim\Phi_k=1 $ for $k=0,\ldots, d-1 $, and the linear maps $U'_j, V'_j $ are defined in \eqref{eq:quiverram2}.
\end{theorem}

\begin{remark}
 In principle, it is possible to give a Jordan-representation of the result also. However, the combinatorics of the various eigenvalues -- in particular additional multiplicities of the eigenvalues $\mu_d $ coming from possible $\beta_j \in \frac{1}{d}\Z $ -- makes it very cumbersome to produce a general (useful) formula. In an explicitly given situation, the computations can be carried out similar to section \ref{sec:Jordanquiver} and \ref{sec:JordanStokes}
\end{remark}

Note that all entries in the representative of Theorem \ref{thm:companionram} arise from products, sums and inverses of matrices the entries of which are polynomial expressions in the coefficients of $\chi_C $ and $\chi_E $. Under the cyclotomic property the latter are integers, and under the conjugate property they are real numbers. Hence we also obtain the analogues of Corollary \ref{cor:cyclo} and \ref{cor:real}:
\begin{corollary}
If the parameters $\ualpha, \ubeta $ both satisfy 
\begin{enumerate}
\item the cyclotomic property, there is a representative with $S_\pm \in \mathrm{GL}_n(\Q) $,
\item the conjugate property, there is a representative with $S_\pm \in \mathrm{GL}_n(\R) $. 
\end{enumerate}
\end{corollary}
 
\subsection*{Acknowledgements}
In a previous version of this paper, we imposed the same extra assumption as Duval-Mitschi for the main result. Christian Sevenheck and Emanuel Scheidegger showed interest in this work and invited me to spend some more thoughts on how to avoid this condition. I am grateful for discussions with them on this subject. I thank the anonymous referees for various suggestions -- especially regarding the extension to the ramified case.

\end{document}